\documentclass[a4paper,reqno]{amsart}

\usepackage[british]{babel}

\usepackage{mlmodern}
\DeclareFontFamily{OMX}{mlmex}{}
\DeclareFontShape{OMX}{mlmex}{m}{n}{<->mlmex10}{}
\usepackage[T1]{fontenc}
\usepackage[babel,stretch=10,shrink=10,verbose=errors,selected]{microtype}
\usepackage[strict,english=british]{csquotes}

\usepackage{amsmath,amssymb,amsthm}
\usepackage{mathtools}
\usepackage[all]{onlyamsmath}

\usepackage[
backend=biber,
style=numeric-comp,
sorting=nyt,
sortcites=true,
maxnames=5,
giveninits=true,
date=year,
useprefix=true,
]{biblatex}
\DeclareFieldFormat{eid}{no.~#1}
\usepackage{enumitem}

\usepackage{tikz}
\usetikzlibrary{cd}

\usepackage{xcolor}
\usepackage[nobiblatex]{xurl}
\usepackage[breaklinks=true,pdflang=en-GB,colorlinks=true]{hyperref}
\colorlet{citecolor}{green!75!black}
\colorlet{linkcolor}{red!75!black}
\colorlet{urlcolor}{blue!75!black}
\hypersetup{citecolor=citecolor,linkcolor=linkcolor,urlcolor=urlcolor}
\usepackage{zref-clever}
\zcsetup{cap,nameinlink=false}
\usepackage{keytheorems}

\usepackage{orcidlink}
\usepackage{ellipsis}

\numberwithin{equation}{section}

\newkeytheorem{theorem}[parent=section,style=plain]
\newkeytheorem{proposition}[sibling=theorem,style=plain]
\newkeytheorem{lemma}[sibling=theorem,style=plain]
\newkeytheorem{corollary}[sibling=theorem,style=plain]

\newkeytheorem{definition}[sibling=theorem,style=definition]
\newkeytheorem{example}[sibling=theorem,style=definition]

\newkeytheorem{remark}[sibling=theorem,style=remark]

\newcommand{\set}[2]{\ensuremath{\{\,{#1}\mid{#2}\,\}}}
\newcommand{\card}[1]{\ensuremath{\left|{#1}\right|}}

\newcommand{\id}{\ensuremath{\mathrm{id}}}
\DeclareMathOperator{\tr}{tr}
\DeclareMathOperator{\diag}{diag}
\DeclareMathOperator{\Span}{span}

\newcommand{\NN}{\mathbb{N}}
\newcommand{\ZZ}{\mathbb{Z}}
\newcommand{\QQ}{\mathbb{Q}}
\newcommand{\RR}{\mathbb{R}}
\newcommand{\CC}{\mathbb{C}}

\DeclareMathOperator{\im}{im}
\DeclareMathOperator{\GL}{GL}

\DeclareMathOperator{\Aut}{Aut}
\DeclareMathOperator{\End}{End}

\DeclareMathOperator{\Fix}{Fix}
\DeclareMathOperator{\Tor}{Tor}
\DeclareMathOperator{\Stab}{Stab}

\newcommand{\normalsub}{\trianglelefteq}

\newcommand{\ind}[2]{\ensuremath{[#1\mathbin{:}#2]}}
\newcommand{\R}{\mathcal{R}}
\newcommand{\restr}[2]{{#1}|_{#2}}
\newcommand{\abs}[1]{\ensuremath{\left|{#1}\right|}}

\newcommand{\filt}[2]{\ensuremath{\mathcal{F}_{#1}(#2)}}
\newcommand{\fact}[2]{\ensuremath{\bar{\mathcal{F}}_{#1}(#2)}}
\newcommand{\comp}[2]{\ensuremath{\mathcal{C}_{#1}(#2)}}
\newcommand{\NR}{\(\mathcal{NR}\)}
\newcommand{\TR}{\ensuremath{\tilde{R}}}

\title{Taming Reidemeister zeta functions on virtually polycyclic groups}
\author[S. Tertooy]{Sam Tertooy\ \orcidlink{0000-0002-5750-9153}}
\date{\today}
\address{KU Leuven, Kulak Kortrijk Campus\\
    E.~Sabbelaan 53\\
    8500 Kortrijk\\
    Belgium}
\email{\href{mailto:sam.tertooy@kuleuven.be}{sam.tertooy@kuleuven.be}}
\urladdr{\url{https://stertooy.github.io}}

\subjclass[2020]{Primary: 20E45; Secondary: 20F19, 55M20}

\keywords{Reidemeister number, Reidemeister zeta function, Nielsen number, Nielsen zeta function, virtually polycyclic group}

\begin{document}
	
	\begin{abstract}
		We prove a formula for Reidemeister numbers of virtually polycyclic groups that generalises known formulas for finite groups and torsion-free virtually polycyclic groups. Based on this formula, we introduce a new invariant which we call the tamed Reidemeister number, and prove that its associated zeta functions are rational on virtually polycyclic groups. We also show that this implies rationality of (ordinary) Reidemeister zeta functions, as well as rationality of Nielsen zeta functions on compact infra-solvmanifolds.
	\end{abstract}
	
	\maketitle

	\section{Introduction}
	
	For a group \(G\) and an endomorphism \(\varphi\), two elements \(g_1,g_2\) are \(\varphi\)-twisted conjugate if
	\[ g_1 = h g_2 \varphi(h)^{-1}\]
	for some \(h \in G\). The equivalence classes are called \(\varphi\)-twisted conjugacy classes and the class of \(g \in G\) is denoted by \([g]_{\varphi}\). The number of \(\varphi\)-twisted conjugacy classes is called the \emph{Reidemeister number} \(R(\varphi)\), and is either a positive integer or infinity. We write \(\R[\varphi]\) for the set of \(\varphi\)-twisted conjugacy classes. If \(R(\varphi^k)\) is finite for all \(k \in \NN\), then \(\varphi\) is called \emph{tame}. For a tame endomorphism, we can define the \emph{Reidemeister zeta function} \(R_\varphi(z)\) as
	\[ R_\varphi(z) \coloneq \exp \sum_{k=1}^\infty\frac{R(\varphi^k)}{k} z^k.\]
	The interest in Reidemeister zeta functions comes from topological fixed-point theory. If \(f\) is a self-map of a closed manifold, then its Nielsen number \(N(f)\) carries certain information on the least number of fixed points in the homotopy class of \(f\). Unfortunately, this number is notoriously hard to compute. If \(M\) is an infra-solvmanifold and \(f_*\) is the endomorphism induced on the fundamental group of \(M\), then \(N(f)\) equals the Reidemeister number \(R(f_*)\) if the latter is finite. The Reidemeister number is, in general, much easier to compute. Nielsen numbers also have a corresponding zeta function, defined as
	\[ N_f(z) \coloneq \exp \sum_{k=1}^\infty\frac{N(f^k)}{k} z^k.\]
	Unlike the Reidemeister zeta function, this definition does not require the induced endomorphism \(f_*\) to be tame, though when it is, both zeta functions coincide.
	
	Dekimpe and Dugardein proved rationality for Nielsen zeta functions of self-maps on infra-nilmanifolds \cite{dd15-a}.  Fel'shtyn and Lee extended this to infra-solvmanifolds of type (R) \cite{fl15-a}. Dekimpe, Van den Bussche and the author proved rationality for Reidemeister zeta functions of certain crystallographic groups in \cite{dtv18-a}, and more recently, Dekimpe and Van den Bussche also proved rationality of Nielsen zeta functions for compact solvmanifolds of dimension at most \(5\), for \NR-solvmanifolds, and for solvmanifolds with fundamental group \(\ZZ^n \rtimes \ZZ\) \cite{dv24-a}. Deré showed that the Reidemeister zeta function of a tame endomorphism of a virtually polycyclic group always equals that of a tame endomorphism of a finitely generated, virtually nilpotent group \cite{dere25-a}. It has been conjectured that both Reidemeister zeta functions on virtually polycyclic groups and Nielsen zeta functions on compact infra-solvmanifolds are rational.
	
	In this paper, we first prove a formula for Reidemeister numbers on (not necessarily torsion-free) virtually polycyclic groups. By replacing Reidemeister numbers in this formula with determinant products on \NR-subgroups, we construct a new invariant we call the tamed Reidemeister number \(\TR(\varphi)\). We show that this new invariant is well-defined, that it equals the Reidemeister number \(R(\varphi)\) when the latter is finite, and that it equals the Nielsen number \(N(f)\) when \(f\) is a self-map on a compact infra-solvmanifold inducing \(\varphi\). We also define a new zeta function corresponding to this invariant,
	\[ \TR_\varphi(z) \coloneq \exp \sum_{k=1}^\infty\frac{\TR(\varphi^k)}{k} z^k,\]
	and show that this function is always rational. As a consequence, we prove the aforementioned conjectures regarding rationality of Reidemeister zeta functions on virtually polycyclic groups and of Nielsen zeta functions on compact infra-solvmanifolds.
	
	This paper is structured as follows. In \zcref{sec:prelims} we go over the group-theoretic preliminaries. In \zcref{sec:Rformula}, we prove a formula for Reidemeister numbers on virtually polycyclic groups, which serves as the inspiration for defining the tamed Reidemeister number in \zcref{sec:TamedRNr}. In \zcref{sec:detaverages} we study the determinant averages needed to prove rationality of tamed Reidemeister zeta functions. This rationality is then proven in \zcref{sec:rationalityZeta}, after which we briefly discuss applications to Reidemeister and Nielsen zeta functions. We conclude with some directions for future work in \zcref{sec:futurework}.
	
	\section{Preliminaries}
	\label{sec:prelims}
	We start by fixing some terminology and conventions. Regarding the natural numbers, we adopt the convention \(\NN \coloneq \{1, 2, \ldots\}\) and \(\NN_0 \coloneq \NN \cup \{0\}\). For a group \(G\) and element \(g\), \(\iota_g\) denotes the inner automorphism \(G \to G \colon x \mapsto gxg^{-1}\). If \(H\) is a subgroup of \(G\), then its isolator is the subset
	\[ \sqrt[G]{H} \coloneq \set{ g \in G }{ \exists k \in \NN \colon g^k \in H}.\]
	We use \(\gamma_i(G)\) to denote the terms of the lower central series, i.e.\@ \(\gamma_1(G) \coloneq G\) and \(\gamma_{i+1}(G) \coloneq [G,\gamma_i(G)] \). If \(G\) is finitely generated, torsion-free and nilpotent, its  \emph{adapted lower central series} is given by \(G_i \coloneq \sqrt[G]{\gamma_i(G)}\). This is a central series with fully invariant terms and whose factors are free abelian.
	
	Empty products and determinants of \(0\)-dimensional matrices both equal \(1\). Finally, following \cite[Not.~8.3]{dv23-a}, we define the next two maps:
	\[
	\abs{\,.\,}_\infty \colon \ZZ \to \NN \cup \{\infty\} \colon x \mapsto \begin{cases}
		\abs{x} & \text{ if } x \neq 0,\\
		\infty &\text{ if } x = 0,
	\end{cases}
	\]
	and
	\[
	\abs{\,.\,}_0 \colon \ZZ \cup \{\infty\} \to \NN_0 \colon x \mapsto \begin{cases}
		\abs{x} & \text{ if } x \neq \infty,\\
		0 &\text{ if } x = \infty.
	\end{cases}
	\]
	
	\subsection{\texorpdfstring{\NR}{NR}-groups and canonical filtrations}
	
	\begin{definition}
		A polycyclic group \(G\) is called \emph{strongly torsion-free} if it fits in a short exact sequence
		\begin{equation}
			\label{eqn:exactseqstfSgrp}
			\begin{tikzcd}
				1 \ar[r]& N \ar[r] & G \ar[r] & \ZZ^n \ar[r] & 1
			\end{tikzcd}
		\end{equation}
		where \(N\) is a finitely generated, torsion-free, nilpotent group.
	\end{definition}
	
	These groups are also called \emph{strongly torsion-free S-groups} in e.g.\@ \cite{km95-a,dv23-a}. The reason behind this is most likely historical: Hirsch \cite{hirs38-a} used the name S-groups for what are now called polycyclic groups. However, we will need a stronger property than strong torsion-freeness: the \NR-property from \cite[Def.~4.1]{dv23-a}.
	
	\begin{definition}
		Let \(G\) be a strongly torsion-free polycyclic group, let \(N \coloneq \sqrt[G]{[G,G]}\) and let \(N = N_1 \geq \cdots \geq N_c \geq N_{c+1} = 1\) be a central series of \(N\) such that each term \(N_i\) is normal in \(G\) and each factor \(N_i/N_{i+1}\) is free abelian. This induces well-defined actions
		\[ \rho_i \colon G \to \Aut(N_i / N_{i+1}) \colon g \mapsto (xN_{i+1} \mapsto \iota_g(x)N_{i+1}).\]
		If for every \(g \in G\) and every \(i\), the automorphism \(\rho_i(g)\) has no non-trivial root of unity as an eigenvalue, then we say \(G\) satisfies the \emph{\NR-property}. We will call a strongly torsion-free polycyclic group \(G\) with the \NR-property an \NR-group.
	\end{definition}
	
	As shown in \cite[Lem.~4.2]{dv23-a}, the \NR-property is well-defined since it is independent of the chosen central series \(N_i\). 
	
	Any finitely generated, torsion-free, nilpotent group is an \NR-group. Indeed, if \(G = G_1 \geq \cdots \geq G_c \geq G_{c+1} = 1\) is its adapted lower central series, then \(N = G_2\) and we can choose \(N_i \coloneq G_{i+1}\). Since the entire sequence \(G \geq N \geq N_2 \geq \cdots \geq N_{c-1} \geq N_c = 1\) is a central series, each \(\rho_i(g)\) will just be the identity map, and hence \(G\) has the \NR-property.
	
	\begin{definition}
		Let \(G\) be an \NR-group, let \(N \coloneq \sqrt[G]{[G,G]}\) and let \(N = N_1 \geq \cdots \geq N_c \geq N_{c+1} = 1\) be the adapted lower central series of \(N\). We define
		\[ \filt{0}{G} \coloneq G, \qquad \filt{i}{G} \coloneq N_i \quad (1 \leq i \leq c+1) ,\]
		and call this the \emph{canonical filtration} of \(G\). The successive quotients
		\[ \fact{i}{G} \coloneq \filt{i}{G} / \filt{i+1}{G}  \quad (0 \leq i \leq c)\]
		are called the \emph{canonical factors}.
	\end{definition}
	This filtration is particularly useful because all of its terms are fully invariant and its factors are free abelian factors. Hence an endomorphism of \(G\) induces an endomorphism on every canonical factor.
	
	\begin{lemma}
		\label{lem:admissiblesubs}
		Let \(G\) be an \NR-group and let \(H\) be a subgroup of \(G\). Then \(H\) has the \NR-property as well.
	\end{lemma}
	\begin{proof}
		Restricting the exact sequence \zcref[nocap,noname]{eqn:exactseqstfSgrp} to \(H\) gives a new exact sequence whose terms still satisfy the requirements of a strongly torsion-free polycyclic group. It remains to check that \(H\) satisfies the \NR-property. By \cite[Lem.~4.2]{dv23-a} we can use any series of normal subgroups of \(H\) with free abelian factors and not just the canonical filtration. After intersecting the canonical filtration of \(G\) with \(H\), i.e.\@ \(H_i \coloneq H \cap \filt{i}{G}\), we find that each resulting factor embeds in the corresponding canonical factor:
		\[
			H_i/H_{i+1} \to \fact{i}{G} \colon xH_{i+1} \mapsto x \filt{i+1}{G}.
		\]
		Consider the map
		\[ \sigma_i \colon H \to \Aut(H_i/H_{i+1}) \colon h \mapsto (xH_{i+1} \mapsto \iota_h(x)H_{i+1}).\]
		For \(i = 0\), the action on \(G/\filt{1}{G}\) is trivial. For \(i \geq 1\), if \(\sigma_i(h)\) were to have an eigenvalue that is a non-trivial root of unity, so would \(\rho_i(h)\), contradicting that \(G\) has the \NR-property. By contraposition, \(H\) has the \NR-property as well.
	\end{proof}
	
	\begin{proposition}
		\label{prop:fullinvsub}
		Let \(G\) be a virtually polycyclic group and let \(H \leq G\) have finite index. Then \(H\) contains an \NR-subgroup \(N\) of finite index that is fully invariant in \(G\).
	\end{proposition}
	\begin{proof}
		The arguments of \cite[Sec.~4.2]{dv23-a} generalise to any virtually polycyclic group \(G\), hence we can always find some finite index \NR-subgroup \(K\) of \(G\). Now set
		\[ n \coloneq \ind{G}{H \cap K}, \qquad N \coloneq \bigcap_{\substack{L \leq G \\ \ind{G}{L} \leq n}} L. \]
		Then \(N\) is fully invariant and has finite index in \(G\) by \cite[Lem.~2.2]{hw03-a}, and is contained in both \(H\) and \(K\). From \zcref{lem:admissiblesubs} it follows that \(N\) has the \NR-property. 
	\end{proof}
	
	\subsection{The determinant invariant \texorpdfstring{\(\tau\)}{τ}}
	
	We introduce a new invariant \(\tau\) on \NR-groups, which serves as the group-theoretic counterpart of the (topological) Nielsen number and the ``always-finite'' counterpart of the Reidemeister number.
	
	\begin{definition}
		\label{def:tau}
		Let \(G\) be an \NR-group and let \(\varphi \in \End(G)\). Let \(\fact{0}{G}\), \ldots, \(\fact{c}{G}\) be its canonical factors and let \(\varphi_i \in \ZZ^{n_i \times n_i}\) be the induced endomorphism on \(\fact{i}{G} \cong \ZZ^{n_i}\). We set
		\[
		\tau(\varphi) \coloneq \prod_{i=0}^c \abs{\det(I_{n_i}-\varphi_i)}.
		\]
	\end{definition}
	Note that while the matrix representation of the \(\varphi_i\) depends on the chosen basis of \(\ZZ^{n_i}\), the determinant \(\det(I_{n_i}-\varphi_i)\) does not. 
	
	\begin{proposition}
		\label{prop:NRtauequalsRwhenfinite}
		Let \(G\) be an \NR-group and let \(\varphi \in \End(G)\). Then \(R(\varphi) = \abs{\tau(\varphi)}_\infty\) and, conversely, \(\tau(\varphi) = \abs{R(\varphi)}_0\).
		
		Moreover, if \(M\) is a compact solvmanifold with fundamental group \(G\), and \(f \colon M \to M\) is a self-map that induces \(\varphi\) on \(G\), then \( N(f) = \tau(\varphi)\).
	\end{proposition}
	\begin{proof}
		The equality for \(N(f)\) follows directly from Keppelmann and McCord's product formula \cite[Thm.~3.1]{km95-a} (see also \cite[Thm.~4.4]{dv23-a}), and the equality for \(R(\varphi)\) follows from the calculations in the proof of \cite[Lem.~8.2]{dv23-a}.
	\end{proof}
	
	\begin{proposition}
		\label{cor:nilpproductformula}
		Let \(G\) be a finitely generated, torsion-free, nilpotent group, and let \(\varphi \in \End(G)\). Let \(G = G_1 \geq \cdots \geq G_c \geq G_{c+1} = 1\) be a central series of \(G\) whose terms are fully invariant in \(G\) and whose factors are free abelian. Denote by \(\varphi_i \in \ZZ^{n_i \times n_i}\) the endomorphisms induced by \(\varphi\) on the factors \(\bar{G}_i \coloneq G_i / G_{i+1} \cong \ZZ^{n_i}\). Then
		\[
		\tau(\varphi) = \prod_{i=1}^c \abs{\det(I_{n_i} - \varphi_i)}.
		\]
	\end{proposition}
	\begin{proof}
		By \cite[Thm.~2.6]{roma11-a} (see also \cite[Sec.~1]{send25-a}), we have that
		\[ R(\varphi) = \prod_{i=1}^c \abs{\det(I_{n_i} - \varphi_i)}_\infty. \]
		Applying \(\abs{\,.\,}_0\) and invoking \zcref{prop:NRtauequalsRwhenfinite} gives the desired formula.
	\end{proof}
	
	\begin{lemma}
		\label{lem:invariances}
		Let \(G,H\) be groups, let \(\varphi \in \End(G)\) and let \(\psi \in \End(H)\). Then
		\begin{enumerate}[label=(\roman*)]
			\item For every \(g \in G\), we have that
			\( R(\iota_g\varphi) = R(\varphi) = R(\varphi\iota_g)\).
			\item If \(\lambda\colon G \to H\) is an isomorphism and \(\psi \lambda = \lambda \varphi\), then
			\( R(\varphi) = R(\psi)\).
		\end{enumerate}
		If \(G\) has the \NR-property, then the same statements hold with \(R\) replaced by \(\tau\).
	\end{lemma}
	\begin{proof}
		For (i), note that the map
		\[ \R[\varphi] \to \R[\iota_g\varphi] \colon  [x]_\varphi \mapsto [xg^{-1}]_{\iota_g\varphi}\]
		is a bijection, and that \(\varphi\iota_g = \iota_{\varphi(g)}\varphi\). Similarly, for (ii), the map
		\[ \R[\varphi] \to \R[\psi] \colon  [x]_\varphi \mapsto [\lambda(x)]_{\psi}\]
		is a bijection. Finally, the assertions for \(\tau\) follow from \zcref{prop:NRtauequalsRwhenfinite}.
	\end{proof}
	
	Restricting an endomorphism of an \NR-group to a finite index, normal subgroup invariant under said endomorphism leaves \(\tau\) unchanged.
	\begin{proposition}
		\label{prop:nilpfiniteindex}
		Let \(G\) be an \NR-group, let \(N \normalsub G\) have finite index, and let \(\varphi \in \End(G)\) with \(\varphi(N) \leq N\). Then for every \(g \in G\), we have that
		\[ \tau(\restr{\iota_g\varphi}{N}) = \tau(\varphi).\]
	\end{proposition}
	\begin{proof}
		To ease notation, we set \(H \coloneq \sqrt[G]{[G,G]}\) and \(M \coloneq \sqrt[N]{[N,N]}\).
		By \zcref{lem:admissiblesubs}, \(N\) has the \NR-property, and then \([N,N]\) has finite index in \([G,G]\) by \cite[Lem.~5.5]{dv23-a}. Consequently, \( M = N \cap H\), and it follows that \(\ind{H}{M} \leq \ind{ G }{ N } < \infty\). Since 
		\[ N/M \cong NH/H \leq G/H,\]
		we can identify \(N/M\) with a subgroup of index \(\ind{G}{NH} < \infty\) in \(G/H\). Let \(\bar{\varphi}\) and \(\tilde{\varphi}\) denote the induced endomorphisms on \(G/H\) and \(N/M\), respectively, and note that we can see \(\tilde{\varphi}\) as the restriction of \(\bar{\varphi}\) to \(N/M\). Since \(H\) and \(G/H\) are both finitely generated, torsion-free, nilpotent groups, we can apply \cite[Thm.~3.19]{send25-a} to obtain
		\[ R(\restr{\varphi}{H}) = R(\restr{\varphi}{M}), \qquad R(\bar{\varphi}) = R(\tilde{\varphi}).\]
		By \zcref{prop:NRtauequalsRwhenfinite}, the same equalities hold with \(R\) replaced by \(\tau\), and by \zcref{def:tau} and \zcref{cor:nilpproductformula} we have that
		\[ \tau(\varphi) = \tau(\restr{\varphi}{H}) \tau(\bar{\varphi}) = \tau(\restr{\varphi}{M}) \tau(\tilde{\varphi}) = \tau(\restr{\varphi}{N}).\]
		Applying this to \(\iota_g\varphi\) and using \zcref{lem:invariances}(i), we obtain that
		\[ \tau(\varphi) = \tau(\iota_g\varphi) = \tau(\restr{\iota_g\varphi}{N}) \]
		for every \(g \in G\).
	\end{proof}
	
	The next few results will be used to prove the averaging formulas in \zcref{prop:averagingformula}, and our new formula for the Reidemeister number in \zcref{thm:Rformula}. The first result is a special case of \cite[Cor.~2.16]{send25-a}: virtually polycyclic groups are finitely generated and residually finite.
	\begin{proposition}
		\label{prop:tfRfiniteFixfinite}
		Let \(G\) be a torsion-free, virtually polycyclic group and let \(\varphi \in \End(G)\). If \(R(\varphi) < \infty\), then \(\Fix(\varphi)\) is trivial.
	\end{proposition}
	
	The following was taken from the author's PhD thesis \cite[Cor.~2.5.15]{tert19-b}.
	\begin{proposition}
		\label{prop:RfiniteFiniteIndex}
		Let \(G\) be a group with finite index normal subgroup \(N\), and let \(\varphi \in \End(G)\) such that \(\varphi(N) \leq N\). Then
		\[ R(\varphi) = \infty \iff \exists gN \in G/N \text{ such that } R(\restr{\iota_g\varphi}{N}) = \infty. \]
	\end{proposition}

	The result below lets us transfer the average of a function on one group to the average of a related function on another group.
	\begin{lemma}
		\label{lem:fingrpaverages}
		Let \(\lambda \colon G \to H\) be an epimorphism of finite groups, and let \(\alpha \colon G \to \CC\) and \(\beta \colon H \to \CC\) be functions such that \(\alpha = \beta\lambda\). Then
		\[ \frac{1}{\card{G}} \sum_{g \in G} \alpha(g) = \frac{1}{\card{H}} \sum_{h \in H} \beta(h). \]
	\end{lemma}
	\begin{proof} This is a straightforward application of the First Isomorphism Theorem:
		\begin{align*}
			\frac{1}{\card{G}} \sum_{g \in G} \alpha(g)
			&= \frac{1}{\card{\lambda(G)}\card{\ker(\lambda)}} \sum_{g \in G} (\beta\lambda)(g) \\
			&= \frac{1}{\card{\lambda(G)}} \sum_{\lambda(g) \in \lambda(G)} \beta(\lambda(g)) \\
			&= \frac{1}{\card{H}} \sum_{h \in H} \beta(h). \qedhere
		\end{align*}
	\end{proof}
	
	The following averaging formulas extend the formulas for Nielsen numbers from \cite[Thm.~3.5]{kll05-a} and \cite[Cor.~4.13]{dv23-a}.
	
	\begin{proposition}
		\label{prop:averagingformula}
		Let \(G\) be a torsion-free, virtually polycyclic group and let \(\varphi \in \End(G)\). Suppose that \(N \normalsub G\) is an \NR-subgroup of finite index and that \(\varphi(N) \leq N\). Then
		\[
		R(\varphi) = \frac{1}{\ind{G}{N}} \sum_{gN \in G/N} \abs{\tau(\restr{\iota_g\varphi}{N})}_\infty.
		\]
		Moreover, if \(G\) is the fundamental group of a compact infra-solvmanifold \(M\) and \(f \colon M \to M\) is a self-map inducing \(\varphi\), then
		\[ N(f) = \frac{1}{\ind{G}{N}} \sum_{gN \in G/N} \tau(\restr{\iota_g\varphi}{N}).\]
	\end{proposition}
	
	\begin{proof}
		The formula for Reidemeister numbers follows from \cite[Prop.~2.5.16]{tert19-b} and \zcref{prop:NRtauequalsRwhenfinite,prop:tfRfiniteFixfinite}.
		
		Using \zcref{prop:fullinvsub}, let \(K \leq N\) be a finite index, fully invariant \NR-subgroup of \(G\).
		Then \cite[Thm.~4.11]{dv23-a} gives us the desired formula, but only for \(K\) (and not for \(N\)), as it requires the normal subgroup to be fully invariant:
		\[
		N(f) = \frac{1}{\ind{G}{K}} \sum_{gK \in G/K} \tau(\restr{\iota_g\varphi}{K}).
		\]
		By \zcref{prop:nilpfiniteindex}, \(\tau(\restr{\iota_{g}\varphi}{K}) = \tau(\restr{\iota_{g}\varphi}{N})\) for every \(g \in G\). Applying \zcref{lem:fingrpaverages} to the natural epimorphism \(G/K \to G/N\) gives the desired formula.
	\end{proof}
	
	\section{A formula for Reidemeister numbers}
	\label{sec:Rformula}
	
	In this section, we express the Reidemeister number as a sum over conjugacy classes of torsion elements; the summands are reminiscent of the averaging formula from \zcref{prop:averagingformula}. We begin by showing that this sum has only finitely many terms.
	
	\begin{proposition}
		\label{prop:segalconjclasses}
		A virtually polycyclic group contains only finitely many conjugacy classes of torsion elements.
	\end{proposition}
	\begin{proof}
		Let \(G\) be a virtually polycyclic group. By \cite[Ch.~8, Thm.~5]{sega83-a}, it has only finitely many conjugacy classes of finite subgroups. Let \(F_1, \ldots, F_n\) be representatives of these classes. Every torsion element in \(G\) generates a finite subgroup and is therefore conjugate to an element of some \(F_i\). Thus, the number of conjugacy classes of torsion elements is bounded above by \(\sum_{i=1}^n \card{F_i} < \infty\).
	\end{proof}

	Let \(G \backslash \Tor(G)\) be the set of \(G\)-conjugacy classes of torsion elements of \(G\). For every \(\varphi \in \End(G)\), we define
	\[
	P_\varphi \colon G \backslash \Tor(G) \to G \backslash \Tor(G) \colon [t]_G \mapsto [\varphi(t)]_G.
	\]
	For a given class \([t]_G \in \Fix(P_\varphi)\), choose a representative \(t\) and pick an element \(g_t \in G\) such that
	\[ t = g_t \varphi(t) g_t^{-1}. \]
	Let \(N\) be a finite index, torsion-free, normal subgroup of \(G\) such that \(\varphi(N) \leq N\). To keep notation short and readable, we set
	\[
	C_t \coloneq C_G(t), \quad D_t \coloneq C_N(t), \quad \varphi_t \coloneq \iota_{g_t}\varphi.
	\]
	\begin{lemma}
		The groups \(C_t\) and \(D_t\) are \(\varphi_t\)-invariant, and \(D_t\) is a finite index, torsion-free, normal subgroup of \(C_t\).
	\end{lemma}
	\begin{proof}
		Let \(c \in C_t\). Then \([\varphi(c),\varphi(t)] = \varphi([c,t]) = \varphi(1) = 1\), thus
		\[
		\varphi_t(c) t
		= g_t \varphi(c) \varphi(t) g_t^{-1} 
		= g_t \varphi(t) \varphi(c) g_t^{-1} 
		= t  \varphi_t(c),
		\]
		which means that \(\varphi_t(c) \in C_t\). All of the required properties of \(D_t\) follow from the fact that it equals \(C_t \cap N\).
	\end{proof}
	
	Below, we state the promised formula.
	
	\begin{theorem}
		\label{thm:Rformula}
		Let \(G\) be a virtually polycyclic group, let \(\varphi \in \End(G)\) and let \(N \normalsub G\) be a \(\varphi\)-invariant, torsion-free, normal subgroup of finite index. Then
		\[
		R(\varphi) = \sum_{[t]_G \in \Fix(P_\varphi)} \frac{1}{\ind{C_t}{D_t}} \sum_{cD_t \in C_t/D_t} R(\restr{\iota_c\varphi_t}{D_t}).
		\]
	\end{theorem}
	\begin{remark}
		If \(G\) is a finite group, we can take \(N = 1\) in \zcref{thm:Rformula} and then \(D_t = 1\) for all \(t\). Thus, the summand of every \([t]_G\) is just
		\[ \frac{1}{\card{C_t}} \sum_{c \in C_t} 1 = 1,\]
		hence we recover a formula discovered by Ado \cite{ado55-a} and Ree \cite{ree59-a}, and later rediscovered by Fel'shtyn and Hill \cite[Thm.~5]{fh94-a}:
		\[ R(\varphi) = \card{\Fix(P_\varphi)} = \card{\set{ [g]_G }{ [g]_G = [\varphi(g)]_G}}. \]
	\end{remark}
	
	\begin{remark}
		\label{rem:formulabecomesavg}
		If \(G\) is torsion-free, then we only need to take the outer sum over a single conjugacy class: that of the identity. Taking \(N\) to be an \NR-subgroup, we get that \(C_1 = G\), \(D_1 = N\), and we can choose \(g_1 = 1\) such that \(\varphi_1 = \varphi\). Using \zcref{prop:NRtauequalsRwhenfinite}, we thus recover the averaging formula from \zcref{prop:averagingformula}.
	\end{remark}
	
	The proof of \zcref{thm:Rformula} will use the following counting formula for group actions. Our only use for it is to apply it to the \(\varphi\)-twisted conjugation action of \(G\) on itself, but it holds in a more general context as well.
	
	\begin{lemma}
		\label{lem:orbitformula}
		Let a group \(G\) act on a set \(X\), and let \(N\) be a finite index, normal subgroup of \(G\). Suppose that the action of \(N\) on \(X\) is free and that the orbit space \(N \backslash X\) is finite. Set
		\[ \mathcal{T} \coloneq \set{ (t,x) \in \Tor(G) \times X }{ t \cdot x = x },\]
		with \(n \cdot (t,x) \coloneq (ntn^{-1}, n \cdot x)\). Then
		\[ \card{G \backslash X} = \frac{\card{N \backslash \mathcal{T}}}{\ind{G}{N}}. \]
	\end{lemma}
	\begin{proof}
		We denote by \(S_x\) the stabiliser \(\Stab_G(x)\) of \(x \in X\). Since \(S_x \cap N\) is trivial, we have that
		\[ S_x = \frac{S_x}{S_x \cap N} \cong \frac{NS_x}{N} \leq \frac{G}{N},\]
		hence \(S_x\) is finite.
		
		The quotient \(G/N\) acts naturally on \(N \backslash X\) via
		\[ G/N \times N \backslash X \to N \backslash X \colon (gN, N \cdot x ) \mapsto N \cdot (g \cdot x). \]
		For a fixed \(N\)-orbit \(N \cdot x\), its stabiliser under this action is exactly \(NS_x/N\), and hence the \(G\)-orbit \(G \cdot x\) is the union of \(\ind{G}{N} /\!\card{S_x}\) \(N\)-orbits. For each such \(N\)-orbit, we fix a representative \(y\). Let \((t',y') \in \mathcal{T}\) such that \(y' \in N \cdot y\). Its \(N\)-orbit contains a unique representative \((t,y)\) with \(t \in \Tor(G)\), because the action of \(N\) on \(X\) is free. There are \(\card{S_y} = \card{S_x}\) such representatives, hence:
		\[ \card{N\backslash \mathcal{T}} 
		= \sum_{G \cdot x \in G \backslash X} \sum_{N \cdot y \subseteq G \cdot x} \card{S_y} 
		= \sum_{G \cdot x \in G \backslash X} \card{S_x} \frac{\ind{G}{N}}{\card{S_x}} 
		= \ind{G}{N} \card{G\backslash X}, \] 
		which after rearranging gives the desired formula.
	\end{proof}
	
	\begin{proof}[Proof of \zcref{thm:Rformula}]
		For \(g,h \in G\), let \(g \ast h\) denote the \(\varphi\)-twisted conjugacy action of \(G\) on itself:
		\[ g \ast h \coloneq g h \varphi(g)^{-1}.\]
		In the first three steps of this proof, we will assume that \(R(\varphi) < \infty\).
		
		\medskip
		
		\noindent\textbf{Step 1.} The restriction of this action to \(N\) is free. Indeed, suppose that for \(g \in G\) and \(n \in N\), we have that \(n \ast g = n g \varphi(n)^{-1} = g\). Then \(n = (\iota_g\varphi)(n)\), so \(n \in \Fix(\restr{\iota_g\varphi}{N})\). Since \(R(\varphi) < \infty\), by \zcref{prop:RfiniteFiniteIndex} also \(R(\restr{\iota_g\varphi}{N}) < \infty\) and by \zcref{prop:tfRfiniteFixfinite}, this fixed point group is trivial. So \(n = 1\). Also, each \(G\)-orbit contains at most \(\ind{G}{N}\) \(N\)-orbits, so the set of \(N\)-orbits is finite. By \zcref{lem:orbitformula}, we have that
		\begin{equation}
			\label{eqn:Rphiorbitaverage}
			R(\varphi) = \frac{\card{N \backslash \mathcal{T}}}{\ind{G}{N}},
		\end{equation}
		where \(\mathcal{T} \coloneq \set{ (t,x) \in \Tor(G) \times G }{ t \ast x = x }\).
		
		\medskip
		
		\noindent\textbf{Step 2.} We now decompose \(N \backslash \mathcal{T}\) according to the conjugacy class of its first component. Set \(Y_t \coloneq \set{x \in G}{t \ast x = x}\). 
		Two pairs \((t,x)\) and \((t,x')\) belong to the same \(N\)-orbit if and only if \(x' = d \ast x\) for some \(d \in D_t\). Thus, we have
		\[ \card{N \backslash \mathcal{T}} = \sum_{[t]_N, Y_t \neq \varnothing} \card{D_t \backslash Y_t}.\]
		Suppose that \(t' = gtg^{-1}\) for some \(g \in G\). The map \(x \mapsto g \ast x\) induces a bijection \(D_t \backslash Y_t \to D_{t'} \backslash Y_{t'}\). Moreover, \(G/N\) acts on the \(N\)-conjugacy classes of torsion elements of \(G\) via
		\[ gN \cdot [t]_N \coloneq [gtg^{-1}]_{N},\]
		and the stabiliser of \([t]_N\) is then
		\[ \frac{NC_t}{N} \cong \frac{C_t}{C_t \cap N} = \frac{C_t}{D_t}. \]
		Thus, by the orbit-stabiliser theorem, \([t]_G\) is the union of \(\frac{\ind{G}{N}}{\ind{C_t}{D_t}}\) \(N\)-conjugacy classes. Since \(Y_t \neq \varnothing\) if and only if \([t]_G \in \Fix(P_\varphi)\), we conclude that
		\begin{equation}
			\label{eqn:NTsumDtYt}
			\card{N \backslash \mathcal{T}} = \sum_{[t]_G \in \Fix(P_\varphi)} \frac{\ind{G}{N}}{\ind{C_t}{D_t}} \card{D_t \backslash Y_t}.
		\end{equation}
		
		\medskip
		
		\noindent\textbf{Step 3.} We now calculate what \(\card{D_t \backslash Y_t}\) is. If \([t]_G = [\varphi(t)]_G\), then \(Y_t\) is the coset \(C_tg_t\). The action of \(D_t\) on \(Y_t\) is given by \(m \ast y = m y \varphi(m)^{-1}\). We define an action of \(D_t\) on \(C_t\):
		\[ d \bullet c \coloneq d c \varphi_t(d)^{-1}.\]
		If \(y_1, y_2 \in Y_t\) and \(c_1, c_2 \in C_t\) satisfy \(y_i = c_ig_t\) for \(i = 1,2\), then \(y_1 = d \ast y_2\) if and only if \(c_1 = d \bullet c_2\), so the number of orbits for both actions is the same. The action of \(D_t\) on \(C_t\) preserves \(D_t\)-cosets. Let \(d, e \in D_t\) and \(c \in C_t\). Then
		\[ d \bullet ec = d ec \varphi_t(d)^{-1} = d e (\iota_c\varphi_t)(d)^{-1}c \in D_t c = cD_t.\]
		The number of \(D_t\)-orbits in \(cD_t\) is therefore \(R( \restr{\iota_c\varphi_t}{D_t})\), and hence
		\begin{equation}
			\label{eqn:DtYtsumReidNrs}
			\card{D_t \backslash Y_t} = \sum_{cD_t \in C_t/D_t} R(\restr{\iota_c\varphi_t}{D_t}).
		\end{equation}
		Combining \zcref[nocap,abbrev]{eqn:Rphiorbitaverage,eqn:NTsumDtYt,eqn:DtYtsumReidNrs} now gives us the desired formula when \(R(\varphi) < \infty\).
		
		\medskip
		
		\noindent\textbf{Step 4.}  Finally, let us consider the case that \(R(\varphi) = \infty\). By \zcref{prop:RfiniteFiniteIndex} we know that \(R(\restr{\iota_g\varphi}{N})\) is infinite for some \(g \in G\). This Reidemeister number will appear in the summand corresponding to the conjugacy class \([1]_G\), so the right-hand side of the equation is infinite as well.
	\end{proof}
	
	\section{The tamed Reidemeister number}
	\label{sec:TamedRNr}
	
	One of the major limitations in working with Reidemeister zeta functions is that we can only define them for tame endomorphisms. In \cite{vanz25-a}, Vanzeir investigated a workaround in his Master's thesis. He defined the \emph{tamed Reidemeister number} of an endomorphism \(\varphi\) as \(\TR(\varphi) \coloneq \abs{R(\varphi)}_0\).
	For this new invariant, the associated zeta functions (\emph{tamed Reidemeister zeta functions}) can always be defined as formal power series. Moreover, the number \(\TR(\varphi)\) carries the same information as the original number \(R(\varphi)\): as the latter cannot be zero, we know that \(\TR(\varphi) = 0 \iff R(\varphi) = \infty\).
	
	Unfortunately, irrational tamed Reidemeister zeta functions \(\TR_\varphi(z)\) were discovered for even relatively ``easy'' groups, e.g.\@ for \(\ZZ^3 \rtimes_{-I} \ZZ_2\). Moreover, if \(M\) is an infra-solvmanifold with a self-map \(f\) inducing an endomorphism \(f_*\) on its fundamental group, then it need not be the case that \(N(f) = \TR(f_*)\). Thus, in this manuscript, we opt for a different definition of the tamed Reidemeister number.
	
	\subsection{Definition and independence of choices}
	
	We limit our scope to virtually polycyclic groups, and inspired by \zcref{thm:Rformula}, we define the tamed Reidemeister number as follows:
	
	\begin{definition}
		\label{def:tamedReidNr}
		Let \(G\) be a virtually polycyclic group, let \(\varphi \in \End(G)\) and let \(N\) be a finite index, normal \NR-subgroup of \(G\) such that \(\varphi(N) \leq N\). We define the \emph{tamed Reidemeister number} \(\TR(\varphi,N)\) as
		\begin{equation}
			\label{eqn:TRformula}
			\TR(\varphi,N) \coloneq \sum_{[t]_G \in \Fix(P_\varphi)} \frac{1}{\ind{C_t}{D_t}} \sum_{cD_t \in C_t/D_t} \tau(\restr{\iota_c\varphi_t}{D_t}).
		\end{equation}
	\end{definition}
	
	Note that such a subgroup \(N\) always exists by \zcref{prop:fullinvsub}.
	
	\begin{proposition}
		\label{prop:Tinvariance1}
		For fixed \(\varphi\) and \(N\), the value of \(\TR(\varphi,N)\) is independent of all choices and hence well-defined:
		\begin{enumerate}
			\item For fixed \(t\) and \(g_t\), the value \(\tau(\restr{\iota_c\varphi_t}{D_t})\) depends only on the coset \(c D_t\) and not on the choice of representative \(c\).
			\item For every \([t]_G\), its summand does not depend on the choice of \(g_t\).
			\item For every \([t]_G\), its summand does not depend on the choice of \(t\).
		\end{enumerate}
	\end{proposition}
	\begin{proof}
		We prove this item by item.
		\begin{enumerate}
			\item Suppose \(c,c' \in C_t\) and \(d \in D_t\) are such that \(c' = dc\). Then
			\[
			\tau( \restr{\iota_{c'} \varphi_t}{D_t}) = \tau( \restr{\iota_d \iota_c \varphi_t}{D_t}) = \tau( \restr{ \iota_c \varphi_t}{D_t})
			\]
			due to \zcref{lem:invariances}(i), so the summand of the coset \(c D_t\) is independent of the chosen representative.
			
			\item Let \(g_t,g'_t\) be such that \(t = g_t\varphi(t)g_t^{-1} = g'_t \varphi(t) (g'_t)^{-1}\). Set \(\varphi'_t \coloneq \iota_{g'_t}\varphi\) and \(c' \coloneq g'_t g_t^{-1}\) such that \(\varphi'_t = \iota_{c'} \varphi_t\). Since \(D_t\) is normal in \(C_t\), the map
			\[
			C_t/D_t \to C_t/D_t \colon c D_t \mapsto c c' D_t
			\]
			is a well-defined bijection, and therefore
			\[
			\sum_{cD_t \in C_t/D_t} \tau(\restr{\iota_c\varphi_t}{D_t}) = \sum_{cD_t \in C_t/D_t} \tau(\restr{\iota_{cc'}\varphi_t}{D_t}) = \sum_{cD_t \in C_t/D_t} \tau(\restr{\iota_{c}\varphi'_t}{D_t}),
			\]
			i.e.\@ the total sum remains the same.
			
			\item Let \(t' \coloneq gtg^{-1}\). We may pick \(g_{t'} \coloneq g g_t \varphi(g)^{-1}\), since
			\[
			g_{t'}\varphi(t')g_{t'}^{-1}
			= g g_t \varphi(t) g_t^{-1} g^{-1}
			= t'.
			\]
			Note that \(C_{t'} = \iota_g(C_t)\), \(D_{t'} = \iota_g(D_t)\), and \(\varphi_{t'}\iota_g = \iota_g\varphi_t\). For \(c \in C_t\), we set \(c' \coloneq gcg^{-1}\). Then \(\iota_g \iota_c = \iota_{c'} \iota_g\) and hence
			\(
			(\iota_{c'}\varphi_{t'} \iota_g)(d) = (\iota_g \iota_c \varphi_t)(d)
			\)
			for every \(d \in D_t\). Thus, by \zcref{lem:invariances}(ii),
			\[ \tau( \restr{\iota_{c'}\varphi_{t'}}{D_{t'}}  ) = \tau( \restr{\iota_c \varphi_t}{D_t} ).\]
			Moreover, the map
			\[ \frac{C_t}{D_t} \to \frac{C_{t'}}{D_{t'}} \colon cD_t \mapsto c' D_{t'}\]
			is a bijection, and in particular \(\ind{C_{t'}}{D_{t'}} = \ind{C_t}{D_t}\). Putting everything together, we obtain
			\[
			\frac{1}{\ind{C_t}{D_t}} \sum_{cD_t \in C_t/D_t} \tau( \restr{\iota_c \varphi_t}{D_t} ) = \frac{1}{\ind{C_{t'}}{D_{t'}}} \sum_{c'D_{t'} \in C_{t'}/D_{t'}} \tau( \restr{\iota_{c'} \varphi_{t'}}{D_{t'}} ),
			\]
			hence the total sum also remains unchanged. \qedhere
		\end{enumerate}
	\end{proof}

	\begin{proposition}
		The value of \(\TR(\varphi,N)\) is independent of the choice of \(N\). 
	\end{proposition}
	\begin{proof}
		First, suppose that \(N' \leq N\) are two finite index, \(\varphi\)-invariant, normal \NR-subgroups. By \zcref{prop:Tinvariance1}, we may use the same representatives \(t\) of the conjugacy classes \([t]_G\) of torsion elements, and the same conjugating elements \(g_t\) for both \(\TR(\varphi,N)\) and \(\TR(\varphi,N')\). Let \(D'_t \coloneq C_{N'}(t)\), which is then a normal subgroup of both \(C_t\) and \(D_t\). In particular, it is \(\iota_c\varphi_t\)-invariant for every \(c \in C_t\).
		
		By \zcref{prop:nilpfiniteindex} we have that
		\[ \tau( \restr{\iota_c\varphi_t}{D'_t} ) = \tau( \restr{\iota_c\varphi_t}{D_t} ) \]
		for every \(c \in C_t\), hence we can apply \zcref{lem:fingrpaverages} to the natural epimorphism \(C_t/D'_t \to C_t/D_t\). Thus, for each conjugacy class \([t]_G\), its summand is the same regardless of whether we used \(N\) or \(N'\), and hence \(\TR(\varphi,N) = \TR(\varphi,N')\).
		
		Second, for arbitrary \(N\) and \(N'\), we can apply the previous step to their intersection \(N \cap N'\):
		\[ \TR(\varphi,N) = \TR(\varphi,N \cap N') = \TR(\varphi,N'),\]
		hence \(\TR(\varphi,N)\) is indeed independent of the specific subgroup \(N\).
	\end{proof}
	
	Henceforth, we write \(\TR(\varphi)\) and omit the ``\(N\)''. We remark that the tamed Reidemeister number will always be a non-negative integer. Non-negativity is obvious from the definition; the proof of integrality will have to wait until the end of \zcref{sec:detaverages}.
	
	\subsection{Comparison with Reidemeister and Nielsen numbers}
	
	We check that \(\TR(\varphi)\) coincides with the Reidemeister number \(R(\varphi)\) when the latter is finite, and with the Nielsen number \(N(f)\) of a self-map \(f\) on a compact infra-solvmanifold which induces \(\varphi\) on the (torsion-free) fundamental group.
	
	\begin{proposition}
		\label{prop:reidistamereidwhenfinite}
		Let \(G\) be a virtually polycyclic group and let \(\varphi \in \End(G)\). If \(R(\varphi) < \infty\), then \(\TR(\varphi) = R(\varphi)\).
	\end{proposition}
	\begin{proof}
		We use \zcref{thm:Rformula} for a suitable subgroup \(N\). If \(R(\varphi) < \infty\), all of the inner summands \(R(\restr{\iota_c\varphi_t}{D_t})\) are finite as well. Thus, by \zcref{prop:NRtauequalsRwhenfinite}, they equal \(\tau(\restr{\iota_c\varphi_t}{D_t})\), and we get exactly the formula for \(\TR(\varphi)\) as stated in \zcref{def:tamedReidNr}.
	\end{proof}
	
	\begin{proposition}
		\label{prop:nielsistamereidwhentf}
		Let \(M\) be a compact infra-solvmanifold with self-map \(f\colon M \to M\), and let \(f_*\) denote the induced endomorphism on the fundamental group. Then \(N(f) = \TR(f_*)\).
	\end{proposition}
	\begin{proof}
		We set \(G \coloneq \pi_1(M)\) and \(\varphi \coloneq f_*\). Using \zcref{prop:fullinvsub}, let \(N\) be a finite index, fully invariant \NR-subgroup of \(G\). As in \zcref{rem:formulabecomesavg}, the formula for \(\TR(\varphi)\) reduces to
		\[ \TR(\varphi) = \frac{1}{\ind{G}{N}} \sum_{gN \in G/N} \tau(\restr{\iota_g\varphi}{N}) = N(f),\]
		where the last equality is given by \zcref{prop:averagingformula}.
	\end{proof}
	
	\section{Determinant averages and associated zeta functions}
	\label{sec:detaverages}
	
	In this section, we consider determinant averages over finite and bounded groups. Our approach follows, in broad strokes, the proof for rationality of Nielsen zeta functions on infra-nilmanifolds in \cite{dd15-a}. In particular, we copied the use of a subgroup of index at most \(2\) to get rid of unwanted absolute values around determinants. However, we diverge from their methods by using compound matrices, rather than Lefschetz numbers and zeta functions, to prove rationality. The use of compound matrices was inspired by the use of exterior powers in \cite{fels91-a}; compound matrices are essentially matrix representations of exterior powers of linear transformations.

	\begin{definition}
		Let \((A_k)_{k \in \NN}\) be a sequence such that \(L \coloneq \limsup_{k \to \infty} \abs{A_k}^{1/k}\) is finite. Its \emph{associated zeta function} is defined as the complex function
		\[ Z_A(z) \coloneq \exp \sum_{k=1}^\infty \frac{A_k}{k} z^k,\]
		which converges when \(\abs{z} < 1/L\) (with the convention that \(1/0 = \infty\)).
	\end{definition}
	
	\begin{proposition}
		\label{prop:generalaverages}
		Let \(G\) be a group, let \(N\) be a finite index, normal subgroup of \(G\), let \(\varphi \in \End(G)\) such that \(\varphi(N) \leq N\), and let \(\rho \colon G \to \GL_n(\ZZ)\) be a homomorphism such that \(\rho(N)\) is abelian. Suppose that \(D\in\ZZ^{n\times n}\) satisfies
		\[ D \rho(g) = \rho(\varphi(g)) D\]
		for every \(g \in G\) and that \(\abs{\det(I-\rho(g)D^k)}\) depends only on \(gN\), for every \(k \geq 1\). Then
		\[ A_k\coloneq\frac{1}{\ind{G}{N}}\sum_{gN\in G/N}\abs{\det(I-\rho(g)D^k)} \]
		is an integer, and its associated zeta function exists and is rational.
	\end{proposition}
	
	We postpone the proof of this \zcref[nocap,noref]{prop:generalaverages} until we have proved the necessary lemmas, which are spread over the following three subsections.
	
	\subsection{Tracing the determinants}
	
	We start by explaining why \(\ZZ\)-linear combinations of traces of matrix powers have rational associated zeta functions. Then, we introduce compound matrices to express determinants as alternating sums of traces, and finally we prove the averaging lemma needed to obtain traces of powers.
	
	The following \zcref[noref,nocap]{lem:plemelj} is often named \emph{Plemelj's formula}. For the reader's convenience, we include a short proof.
	\begin{lemma}
		\label{lem:plemelj}
		Let \(D \in \CC^{n \times n}\) with eigenvalues \(\lambda_1, \ldots, \lambda_n\) and let \(z \in \CC\) with \(\abs{z\lambda_i} < 1\) for all \(i = 1, \ldots, n\). Then
		\[ \det(I_n-zD) = \exp \left( - \sum_{k=1}^\infty \frac{ \tr(D^k)}{k} z^k\right). \]
	\end{lemma}
	\begin{proof}
		By using the Jordan normal form, it is clear that
		\[ \det(I_n - zD) = \prod_{i=1}^n (1-z\lambda_i), \qquad \tr(D^k) = \sum_{i=1}^n \lambda_i^k.\]
		Using the power series expansion \(\log(1-w) = -\sum_{k = 1}^\infty w^k/k\) for \(\abs{w} < 1\), we obtain
		\begin{align*}
			\det(I_n - zD)
			&= \exp \left( \sum_{i=1}^n \log(1-z\lambda_i) \right) \\
			&= \exp \left( - \sum_{k=1}^\infty \frac{z^k}{k} \sum_{i=1}^n  \lambda_i^k \right)
			= \exp \left( - \sum_{k=1}^\infty \frac{\tr(D^k)}{k} z^k \right),
		\end{align*}
		for any \(z \in \CC\) with \(\abs{z} < 1/\abs{\lambda_i}\) for all \(i = 1, \ldots,n\) (with the convention that \(1/0 = \infty\)).
	\end{proof}
	
	\begin{corollary}
		\label{cor:zetafuncoftraces}
		Let \(M_i \in \CC^{n_i \times n_i}\) for \(i = 1, \ldots, r\) and let \(c_1, \ldots, c_r \in \ZZ\) be constants. The sequence \((A_k)_{k \in \NN}\) with
		\[ A_k \coloneq \sum_{i=1}^r c_i \tr(M_i^k) \]
		has a rational associated zeta function \(Z_A(z)\), given by
		\[
		Z_A(z) = \prod_{i=1}^r \det(I_{n_i} - z M_i)^{- c_i}.
		\]
	\end{corollary}
	\begin{proof}
		Let \(C \geq 1\) be a constant greater than \(\abs{\lambda}\), for every eigenvalue \(\lambda\) of every matrix \(M_i\). Then for every \(k \geq 1\), we have that
		\[ \abs{A_k} \leq \sum_{i=1}^r \abs{c_i} n_i C^k, \]
		hence \(Z_A(z)\) exists. The expression in terms of determinants is a consequence of Plemelj's formula, and rationality follows from the fact that these determinants are polynomials in \(z\).
	\end{proof}
	
	We now introduce compound matrices and some of their properties.
	
	\begin{definition}
		Let \(D \in \CC^{m \times n}\). Let \(I \subseteq \{1,\ldots,m\}\) and \(J\subseteq \{1,\ldots,n\}\) be index sets of \(0 \leq r \leq \min\{m,n\}\) elements. The \emph{\(r\)-th compound matrix} \(\comp{r}{D}\) of \(D\) is an \(\binom{m}{r} \times \binom{n}{r}\) matrix whose \((I,J)\)-th entry is the minor obtained by retaining the \(r\) rows and columns whose indices appear in \(I\) and \(J\), respectively. The index sets are ordered lexicographically. For \(r = 0\), we set \(\comp{0}{D} \coloneq (1)\).
	\end{definition}
	
	It is worth noting that if \(D\) is a rational (or integer) matrix, then so are all its compound matrices. A very useful property of compound matrices is that they are multiplicative \cite[Eq.~(0.8.1.1)]{hj12-a}:
	
	\begin{lemma}
		Let \(A\) be an \(m \times k\) and \(B\) a \(k \times n\) matrix. For every \(r \leq \min\{m,k,n\}\), we have
		\[ \comp{r}{AB} = \comp{r}{A} \comp{r}{B}.\]
	\end{lemma}
	
	The following is a special case of \cite[Eq.~(0.8.12.3)]{hj12-a}.
	\begin{lemma}
		\label{lem:detalttracesum}
		Let \(D \in \CC^{n \times n}\). Then
		\[ \det(I_n - D) = \sum_{r=0}^n (-1)^r \tr(\comp{r}{D}).\]
	\end{lemma}
	
	The following \zcref[noref,nocap]{lem:integerbase} lets us choose a \(\ZZ\)-basis of \(\ZZ^n\) adapted to a series of \(\QQ\)-vector subspaces. 
	\begin{lemma}
		\label{lem:integerbase}
		Let \(0 = V_0 \subseteq V_1 \subseteq \cdots \subseteq V_r = \QQ^n\) be a series of vector subspaces. There is a \(\ZZ\)-basis of \(\ZZ^n\) whose first \(\dim V_i\) vectors span \(V_i\) over \(\QQ\) for every \(i\). With respect to this basis, every integer matrix that leaves every \(V_i\) invariant is block upper triangular, and its (integer) diagonal blocks describe its actions on the quotients \(V_i / V_{i-1}\). Moreover, if its inverse is also an integer matrix and leaves each \(V_i\) invariant, then the inverses of the diagonal blocks are integer matrices as well.
	\end{lemma}
	\begin{proof}
		For every \(i\), we put \(\Lambda_i \coloneq V_i \cap \ZZ^n\). Pick a \(\QQ\)-basis of \(V_i\), and multiply each of the basis vectors by a positive integer such that it belongs to \(\Lambda_i\) (e.g.\@ the least common multiple of the denominators of its entries). This new basis still spans all of \(V_i\) over \(\QQ\), so the \(\QQ\)-span of \(\Lambda_i\) is \(V_i\).
		
		The quotients \(\Lambda_i/\Lambda_{i-1}\) are free abelian: if \(v \in \Lambda_i\) and for some \(r \in \NN\) we have that \(rv \in V_{i-1}\), then also \(v \in V_{i-1}\) and hence \(v \in \Lambda_{i-1}\).
		
		For each \(i\) we pick a basis of \(\Lambda_i/\Lambda_{i-1}\) and then pick preimages in \(\Lambda_i\) of the basis vectors. All of these preimages together provide the required basis of \(\Lambda_r = \ZZ^n\).
	\end{proof}
	
	\begin{lemma}
		\label{lem:traceavgs}
		Let \(G\) be a finite group, let \(\varphi \in \End(G)\), let \(\rho \colon G \to \GL_n(\RR)\) be a homomorphism and let \(D \in \RR^{n\times n}\) such that \(D \rho(g) = \rho(\varphi(g)) D\) for every \(g \in G\). Then there exists a real matrix \(M\) such that, for every \(k \geq 1\),
		\[
		\frac{1}{\card{G}} \sum_{g \in G} \tr(\rho(g)D^k) = \tr(M^k).
		\]
		If moreover \(D\) and \(\rho(g)\) (for every \(g \in G\)) are integer matrices, then \(M\) can be chosen to be an integer matrix as well.
	\end{lemma}
	\begin{proof}
		Let \(P\) be the matrix
		\[ P \coloneq \frac{1}{\card{G}} \sum_{g \in G} \rho(g). \]
		For any \(h \in G\), multiplying \(P\) and \(\rho(h)\) just permutes the summands, hence
		\[ P\rho(h) = P, \qquad P^2 = P.\]
		Combining this with \(D \rho(g) = \rho(\varphi(g)) D\), we find
		\[
		PDP
		= \frac{1}{\card{G}} \sum_{g \in G} PD\rho(g)
		= \frac{1}{\card{G}} \sum_{g \in G} P\rho(\varphi(g)) D
		= \frac{1}{\card{G}} \sum_{g \in G} PD
		= PD,
		\]
		from which we conclude that \(\ker P\) is \(D\)-invariant. Any vector \(v \in \RR^n\) can be written as
		\[ v = (v-Pv) + Pv \in \ker P \oplus \im P,\]
		hence with respect to a suitable basis we have
		\[ P = \begin{pmatrix}
			0 & \ast\\
			0 & I
		\end{pmatrix},
		\qquad
		D = \begin{pmatrix}
			\ast & \ast \\
			0 & M
		\end{pmatrix}\]
		for some matrix \(M\). Therefore
		\[
		\frac{1}{\card{G}} \sum_{g \in G} \tr(\rho(g)D^k) = \tr(PD^k) = \tr(M^k),
		\]
		which shows that \(M\) provides the desired matrix. If \(D\) and all the \(\rho(g)\) are integer matrices, then \(P\) is a rational matrix. By choosing the basis over \(\QQ\) as in \zcref{lem:integerbase}, \(M\) will be an integer matrix.
	\end{proof}
	
	\subsection{Reduction to bounded representations}
	
	In order to prove \zcref{prop:generalaverages}, we need \(\rho\) and \(D\) to be sufficiently ``well-behaved''. By using the assumption that \(\rho(N)\) is abelian, we will first show that we may assume \(D\) is invertible and
	\(\rho(N)\) is simultaneously diagonalisable. We will then construct \(\tilde{\rho}\) and \(\tilde{D}\) such that \(\tilde{\rho}(G)\) is bounded, without changing the determinants or the relation \(D \rho(g) = \rho(\varphi(g)) D\).
	
	\begin{lemma}
		\label{lem:matrixreduction}
		Let \(G\) be a group, let \(N\) be a finite index, normal subgroup of \(G\), let \(\varphi \in \End(G)\) such that \(\varphi(N) \leq N\), and let \(\rho\colon G \to \GL_n(\ZZ)\) be a homomorphism such that \(\rho(N)\) is abelian. Suppose that \(D \in \ZZ^{n\times n}\) satisfies \(D \rho(g)=\rho(\varphi(g)) D\) for every \(g \in G\). Then there exist a homomorphism \(\tilde{\rho}\colon G \to \GL_m(\ZZ)\) and a matrix \(\tilde{D} \in \ZZ^{m \times m}\) with non-zero determinant (with \(m \leq n\)) such that \(\tilde{\rho}(N)\) is simultaneously diagonalisable over \(\CC\) and
		\begin{equation}
			\label{eqn:tildematsrels}
			\tilde{D} \tilde{\rho}(g) = \tilde{\rho}(\varphi(g)) \tilde{D},
			\qquad
			\det(I_n-\rho(g)D^k) = \det(I_m-\tilde{\rho}(g)\tilde{D}^k)
		\end{equation}
		for every \(g \in G\) and every \(k \geq 1\).
	\end{lemma}
	\begin{proof}
		The proof goes in three steps.
		
		\medskip 
		
		\noindent\textbf{Step 1.} In this first step, we pass to a quotient on which \(D\) is invertible. Set \(V \coloneq \QQ^n\) and choose \(r \in \NN\) such that \(W \coloneq \ker D^r = \ker D^{r+1}\). Since \(D^j \rho(g) = \rho(\varphi^j(g)) D^j\), every \(\ker D^j\) is \(\rho(G)\)-invariant. Moreover, \(\rho(g)D^k\) maps \(\ker D^j\) into \(\ker D^{j-1}\), for every \(j,k \geq 1\), so its restriction to \(W\) is nilpotent. Let \(\bar{D}\) and \(\bar{\rho}\) be the induced maps for the quotient \(\bar{V} \coloneq V/W\), and set \(m \coloneq \dim \bar{V}\). With respect to a suitable basis,
		\[
		\rho(g)D^k = \begin{pmatrix}
			\restr{\rho(g)D^k}{W} & \ast \\
			0 & \bar{\rho}(g)\bar{D}^k
		\end{pmatrix}.
		\]
		Since \(\restr{\rho(g)D^k}{W}\) is nilpotent, we get that
		\[ \det(I_n - \rho(g)D^k) = \det(I- \restr{\rho(g)D^k}{W})\det(I_m - \bar{\rho}(g)\bar{D}^k) = \det(I_m - \bar{\rho}(g)\bar{D}^k).\]
		The matrix \(\bar{D}\) is now invertible: if \(\bar{D}(v+W) = 0\), then \(Dv \in W = \ker D^r\), hence \(v \in \ker D^{r+1} = W\). By \zcref{lem:integerbase}, \(\bar{D}\) and \(\bar{\rho}(g)\) are still integer matrices, and \(\bar{D} \bar{\rho}(g) = \bar{\rho}(\varphi(g)) \bar{D}\) still holds.
		
		\medskip
		
		\noindent\textbf{Step 2.} In this second step, we construct a series of invariant subspaces such that the matrices induced by \(\bar{\rho}(N)\) on each successive quotient are simultaneously diagonalisable. Let \(\mathcal{A}\) be the \(\QQ\)-vector space spanned by \(\bar{\rho}(N)\). Since \(\bar{\rho}(N)\) is an abelian group, \(\mathcal{A}\) contains the identity matrix, is closed under multiplication, and its elements commute pairwise. Normality of \(N\) and the matrix relation give us that
		\[
		\bar{\rho}(g) \mathcal{A} \bar{\rho}(g)^{-1} = \mathcal{A},
		\qquad
		\bar{D} \mathcal{A} \bar{D}^{-1}=\mathcal{A},
		\]
		where the second equality follows from the inclusion \(\bar{D} \mathcal{A} \bar{D}^{-1} \subseteq \mathcal{A}\) with both spaces having equal dimension. Denote by \(\mathcal{B}\) the set of nilpotent matrices in \(\mathcal{A}\), and set
		\[
		\bar{V}_0 \coloneq 0,
		\qquad
		\bar{V}_{i+1} \coloneq \set{ v\in \bar{V} }{ Bv \in \bar{V}_i \text{ for every } B \in \mathcal{B}}.
		\]
		All of the \(\bar{V}_i\) are invariant under \(\mathcal{A}\), \(\bar{\rho}(G)\), \(\bar{D}\) and \(\bar{D}^{-1}\). This can be shown by induction: the matrices in \(\mathcal{A}\) commute, and conjugation by \(\bar{D}\) or \(\bar{\rho}(g)\) permutes \(\mathcal B\). In particular, \(B(\bar{V}_i) \subseteq \bar{V}_i\) for every \(B \in \mathcal{B}\), hence \(\bar{V}_i \subseteq \bar{V}_{i+1}\).
		
		We now show that this sequence eventually reaches \(\bar{V}\). The space \(\bar{V}_i\) contains exactly those vectors that are mapped to zero by any product of \(i\) matrices from \(\mathcal{B}\). Since \(\mathcal{A}\) is finite-dimensional, we can choose \(B_1, \ldots, B_s \in \mathcal{B}\) such that every matrix in \(\mathcal{B}\) is a \(\QQ\)-linear combination of these matrices. Note that every \(B_j\) is nilpotent on \(\bar{V}\), hence \(B_j^m = 0\).
		
		Every product of \(sm\) \(B_j\)'s must contain an \(m\)-th (or larger) power of some \(B_j\), and since the \(B_j\)'s commute, this product is then the zero matrix. Since \(\mathcal{B}\) is contained in the \(\QQ\)-span of the \(B_j\)'s, this also holds for every product of \(sm\) matrices from \(\mathcal{B}\). Hence, we find that \(\bar{V}_{sm} = \bar{V}\).
		
		Let \(M \in \bar{\rho}(N)\) and let \(p \in \QQ[z]\) be the product of the distinct monic irreducible factors of its characteristic polynomial. Then \(p(M) \in \mathcal{A}\), and since the characteristic polynomial divides a sufficiently high power of \(p\), the Cayley-Hamilton theorem \cite[Thm.~2.4.3.2]{hj12-a} gives that \(p(M)^q = 0\) for some \(q \geq 1\). Thus \(p(M) \in \mathcal{B}\), and the matrix \(M'_i\) induced by \(M\) on \(\bar{V}_i/\bar{V}_{i-1}\) satisfies \(p(M'_i) = 0\). Since \(p\) has no repeated roots, \(M'_i\) is diagonalisable by \cite[Cor.~3.3.10]{hj12-a}. The induced matrices commute, hence they are simultaneously diagonalisable by \cite[Thm.~1.3.21]{hj12-a}.

		\medskip
		
		\noindent\textbf{Step 3.} In this third and final step, we replace \(\bar{D}\) and \(\bar{\rho}(g)\) by their diagonal blocks, and prove that the resulting matrices have the required properties. Use \zcref{lem:integerbase} to choose a \(\ZZ\)-basis adapted to the subspaces \(\bar{V}_i\). Let \(\tilde{D}\) be the block diagonal matrix whose block diagonal equals that of \(\bar{D}\), and define \(\tilde{\rho}\) similarly. Then \zcref[nocap,abbrev]{eqn:tildematsrels} holds for every \(g \in G\) and every \(k \geq 1\). Moreover, \(\tilde{D}\) is an integer matrix invertible over \(\QQ\), \(\tilde{\rho}(g)\) and \(\tilde{\rho}(g)^{-1}\) are integer matrices, and \(\tilde{\rho}(N)\) is simultaneously diagonalisable.
	\end{proof}
	
	Next, we introduce the positive part of a matrix. This is a real matrix which retains information on the moduli of the eigenvalues of the original matrix. By ``removing'' the positive part from every matrix in a simultaneously diagonalisable group, we will be left with a bounded group.
	
	\begin{definition}
		Let \(M \in \RR^{n \times n}\) be diagonalisable over \(\CC\). If \(P \in \GL_n(\CC)\) is such that \(M = P \diag(\lambda_1, \ldots, \lambda_n) P^{-1}\), we define the \emph{positive part} of \(M\) as the real matrix
		\[ M_+ \coloneq P \diag(\abs{\lambda_1}, \ldots, \abs{\lambda_n}) P^{-1},\]
		such that \(M_+\) acts as multiplication by \(\abs{\lambda_i}\) on the \(\lambda_i\)-eigenspace of \(M\).
	\end{definition}
	
	\begin{lemma}
		\label{lem:positiverescaling}
		Let \(M \in \GL_n(\RR)\) be diagonalisable over \(\CC\), and let \(D \in \RR^{n \times n}\). If the sequence \((\det(I_n - M^k D))_{k \in \ZZ}\) is bounded, then
		\[
		\det(I_n - M_+^t D) = \det(I_n - D)
		\]
		for every \(t \in \RR\).
	\end{lemma}
	\begin{proof}
		We will work in an eigenbasis of \(M\), and denote the eigenvalues by \(\lambda_1, \ldots, \lambda_n\). By \zcref{lem:detalttracesum} and multiplicativity of compound matrices,
		\begin{align*}
			\det(I_n - M^t D)
			&= \sum_{r=0}^n (-1)^r \tr \left( \comp{r}{M}^t \comp{r}{D}\right), \\
			\det(I_n - M_+^t D)
			&= \sum_{r=0}^n (-1)^r \tr \left( \comp{r}{M_+}^t \comp{r}{D}\right).
		\end{align*}
		The diagonal entries of \(\comp{r}{M}\) are the products \(\prod_{i\in J}\lambda_i\), where \(\card{J}=r\), and the corresponding entries of \(\comp{r}{M_+^t}\) are their absolute values raised to the \(t\)-th power. Let \(\mu_1, \ldots, \mu_s\) be the distinct products occurring for \(0 \leq r \leq n\), including the empty product \(1\). Grouping equal products gives
		\[
		\det(I_n - M^t D) =\sum_{j=1}^s c_j \mu_j^t,
		\qquad
		\det(I_n - M_+^tD) = \sum_{j=1}^s c_j \abs{\mu_j}^t,
		\]
		where \(c_j\) is the sum of the diagonal entries of \((-1)^r \comp{r}{D}\) whose corresponding eigenvalue product is \(\mu_j\). Thus, both sums share these coefficients.
		
		The Vandermonde matrix \(V \coloneq (\mu_j^{i-1})_{1 \leq i,j \leq s}\) is invertible because all of the \(\mu_j\) are distinct. Hence
		\[
		\begin{pmatrix}
			c_1 \mu_1^t\\
			\vdots\\
			c_s \mu_s^t
		\end{pmatrix}
		= V^{-1}
		\begin{pmatrix}
			\det(I_n - M^t D)\\
			\vdots\\
			\det(I_n - M^{t+s-1} D)
		\end{pmatrix}
		\]
		for every \(t \in \ZZ\). The right-hand side is bounded, independently of the value of \(t\). If \(\abs{\mu_j}>1\), letting \(t \to +\infty\) forces \(c_j=0\). Similarly, if \(\abs{\mu_j}<1\), we can use \(t \to -\infty\). The only remaining terms are those with \(\abs{\mu_j}=1\), so
		\[
		\det(I_n - M_+^t D) = \sum_{j=1}^s c_j = \det(I_n - D). \qedhere
		\]
	\end{proof}
	
	\begin{lemma}
		\label{lem:imginvariance}
		Let \(G\) be a group, let \(\varphi \in \End(G)\), and let \(\rho \colon G \to \GL_n(\RR)\) be a homomorphism such that \(\rho(G)\) is simultaneously diagonalisable over \(\CC\). Suppose that \(D \in \GL_n(\RR)\) satisfies \(D\rho(g) = \rho(\varphi(g))D\) for every \(g \in G\). Then
		\[ \rho(\varphi(G)) = \rho(G). \]
	\end{lemma}
	\begin{proof}
		Let \(V\) be the \(\RR\)-vector space spanned by \(\rho(G)\). The matrix relation gives us that
		\[ D\rho(G)D^{-1} = \rho(\varphi(G)) \subseteq \rho(G), \]
		hence \(DVD^{-1} \subseteq V\). Since both spaces have the same dimension, we get that \(DVD^{-1} = V\). Let \(W_1, \ldots, W_r\) be the common eigenspaces of \(\rho(G)\) over \(\CC\). Every matrix in \(V\) acts as scalar multiplication on each \(W_i\). We fix an eigenspace \(W_i\) and a \(g \in G\). Since \(D^{-1}\rho(g)D \in V\), there exists a \(\lambda_{g,i}\) such that \((D^{-1}\rho(g)D)w = \lambda_{g,i} w\) for every \(w \in W_i\). Moreover,
		\[ \rho(g)Dw = D(D^{-1}\rho(g)D)w = \lambda_{g,i} Dw, \]
		so \(DW_i\) is contained in a common eigenspace \(W_j\). Applying the same argument to \(D^{-1}\), we find that \(D\) permutes the common eigenspaces \(W_i\). A suitable power \(D^t\) therefore leaves each eigenspace \(W_i\) invariant and commutes with \(\rho(G)\). Therefore,
		\[
		\rho(G)
		\supseteq D \rho(G) D^{-1}
		\supseteq \cdots
		\supseteq D^t \rho(G) D^{-t}
		= \rho(G),
		\]
		so all inclusions are in fact equalities. Thus,
		\[ \rho(\varphi(G)) = D \rho(G) D^{-1} = \rho(G). \qedhere \]
	\end{proof}
	
	\begin{lemma}
		\label{lem:positivenormalisation}
		Let \(A \leq \GL_n(\RR)\) be simultaneously diagonalisable over \(\CC\), and let \(N\) be its normaliser in \(\GL_n(\RR)\). Let \(\mathfrak{A} \coloneq \langle a_+^t \mid a \in A,\, t \in \RR \rangle\). There exists a homomorphism \(\pi \colon N \to \GL_n(\RR)\) such that
		\[
		\pi(a) = a_+^{-1}a,
		\qquad
		\pi(h)h^{-1} \in \mathfrak{A},
		\]
		for all \(a \in A\) and all \(h \in N\).
	\end{lemma}
	\begin{proof}
		Let \(V_1, \ldots, V_r\) be the common eigenspaces of \(A\) and set \(d_i \coloneq \dim_{\CC} V_i\). We will work in an eigenbasis corresponding to these eigenspaces. For each \(i = 1, \ldots, r\), we define \(\lambda_i \colon A \to \CC^\ast\) such that \(\restr{a}{V_i} = \lambda_i(a)I\). We set
		\[
		\ell(a) \coloneq (\log\abs{\lambda_i(a)})_{i=1}^r \in \RR^r,
		\qquad
		W \coloneq \Span_{\RR}\ell(A).
		\]
		To each \(x \in \RR^r\), we associate a matrix \(M_x\). With respect to an eigenbasis corresponding to the eigenspaces \(V_i\), this matrix is diagonal such that its entries corresponding to \(V_i\) are \(e^{x_i}\). Clearly, \(M_x M_y = M_{x+y}\) for all \(x,y \in \RR^r\). Since
		\[
		M_{\sum_j t_j\ell(a_j)}=\prod_j(a_j)_+^{t_j},
		\]
		we have that \(\mathfrak{A} = \set{M_x}{x\in W}\), and every \(M_x\) with \(x \in W\) is real.
		
		Each element \(h \in N\) induces a permutation \(\sigma_h\) of the common eigenspaces, given by \(h V_i = V_{\sigma_h(i)}\). We also get a linear map \(\mathcal{S}_h\) on \(\RR^r\) such that \((\mathcal{S}_h x)_{\sigma_h(i)} = x_i\) for every \(x \in \RR^r\). Then
		\[
		hM_xh^{-1}=M_{\mathcal{S}_h x},
		\qquad
		\ell(hah^{-1}) = \mathcal{S}_h \ell(a),
		\]
		thus \(W\) is invariant under \(\mathcal{S}_h\). Denote by \(\mathcal{P} \colon \RR^r \to W\) the orthogonal projection. We have that \(\mathcal{P} \mathcal{S}_h = \mathcal{S}_h \mathcal{P}\).
		
		We now define a map \(\mathcal{D} \colon N \to \RR^r\) given by
		\[ \mathcal{D}(h)_{\sigma_h(i)} \coloneq \frac{1}{d_i} \log \abs{ \det(h[\sigma_h(i),i])  }, \]
		where \(h[\sigma_h(i),i]\) is the \(d_i \times d_i\) block of \(h\) corresponding to how \(h\) maps \(V_i\) to \(V_{\sigma_h(i)}\).
		
		If we set \(j \coloneq \sigma_{h'}(i)\) and \(k \coloneq \sigma_h(j)\), then the block of \(hh'\) mapping \(V_i\) to \(V_k\) is the product of the block of \(h\) mapping \(V_j\) to \(V_k\) and the block of \(h'\) mapping \(V_i\) to \(V_j\), i.e.\@ \((hh')[k,i] = h[k,j]h'[j,i]\). Therefore:
		\[
		\mathcal{D}(hh') = \mathcal{D}(h) + \mathcal{S}_h \mathcal{D}(h').
		\]
		For \(a\in A\), we have \(\det(\restr{a}{V_i}) = \lambda_i(a)^{d_i}\), hence \(\mathcal{D}(a)=\ell(a)\). We can finally define the required morphism:
		\[ \pi \colon N \to \GL_n(\RR) \colon h \mapsto M_{-\mathcal{P}\mathcal{D}(h)} h.\]
		Essentially, this removes only the part of the logarithmic determinants lying in \(W\). Using the identities above, we get
		\begin{align*}
			\pi(h)\pi(h')
			&= M_{-\mathcal{P} \mathcal{D}(h)} M_{-\mathcal{S}_h \mathcal{P} \mathcal{D}(h')}hh'\\
			&=M_{-\mathcal{P}( \mathcal{D}(h) + \mathcal{S}_h \mathcal{D}(h') )} h h'
			=\pi(hh').
		\end{align*}
		Thus, \(\pi\) is a homomorphism. Since \(\mathcal{P}\ell(a)=\ell(a)\), we get that \(\pi(a)=M_{-\ell(a)}a=a_+^{-1}a\), and \(\pi(h)h^{-1}=M_{-\mathcal{P}\mathcal{D}(h)} \in \mathfrak{A}\).
	\end{proof}
	
	\begin{lemma}
		\label{lem:boundedreduction}
		Let \(G\) be a group, let \(N\) be a finite index, normal subgroup of \(G\), let \(\varphi \in \End(G)\) such that \(\varphi(N) \leq N\), and let \(\rho\colon G \to \GL_n(\RR)\) be a homomorphism such that \(\rho(N)\) is simultaneously diagonalisable over \(\CC\). Suppose that \(D \in \GL_n(\RR)\) satisfies \(D \rho(g)=\rho(\varphi(g)) D\) for every \(g \in G\). If \(\abs{\det(I_n - \rho(g)D^k)}\) depends only on the coset \(gN\), for every \(k \geq 1\), then there exists a homomorphism \(\tilde{\rho} \colon G \to \GL_n(\RR)\) and a matrix \(\tilde{D} \in \GL_n(\RR)\) such that \(\tilde{\rho}(G)\) is bounded and
		\[
			\tilde{D} \tilde{\rho}(g) = \tilde{\rho}(\varphi(g)) \tilde{D},
			\qquad
			\det(I_n-\rho(g)D^k) = \det(I_n-\tilde{\rho}(g)\tilde{D}^k)
		\]
		for every \(g \in G\) and every \(k \geq 1\).
	\end{lemma}
	\begin{proof}
		Set \(A \coloneq \rho(N)\).  We apply \zcref{lem:imginvariance} to \(N\), \(\restr{\varphi}{N}\), \(\restr{\rho}{N}\) and \(D\), which gives us that \(DAD^{-1} = \rho(\varphi(N)) = A\). Since \(N\) is normal in \(G\), each matrix \(\rho(g)\) also normalises \(A\). Therefore, we may apply \zcref{lem:positivenormalisation} to \(A = \rho(N)\), and define
		\[
		\tilde{\rho} \coloneq \pi \rho,
		\qquad
		\tilde{D} \coloneq \pi(D).
		\]
		Since \(\pi\) is a homomorphism when restricted to the normaliser of \(A\), we have that
		\[
		\tilde{D} \tilde{\rho}(g) = \pi(D \rho(g)) =  \pi(\rho(\varphi(g)) D) = \tilde{\rho}(\varphi(g)) \tilde{D},
		\qquad
		\tilde{\rho}(g)\tilde{D}^k=\pi(\rho(g)D^k).
		\]
		In a common eigenbasis of \(A\), the matrices \(\pi(a) = a_+^{-1}a\) are diagonal with eigenvalues of modulus one, hence \(\tilde{\rho}(N) = \pi(A)\) is bounded. The group \(\tilde{\rho}(G)\) is a finite union of translates of \(\tilde{\rho}(N)\), so it is bounded as well.
		
		It remains to check the determinants. Fix \(g \in G\) and \(k \geq 1\). For every \(a\in A\), choose an \(x \in N\) such that \(a = \rho(x)\). For every \(t \in \ZZ\), normality gives us that \(x^t g N = g N\), and so the hypothesis gives
		\[
		\abs{\det(I_n-a^t \rho(g) D^k)} = \abs{\det(I_n - \rho(g) D^k)}.
		\]
		We now invoke \zcref{lem:positiverescaling} to find that
		\begin{equation}
			\label{eqn:elemposparts1}
			\det(I_n - a_+^t \rho(g) D^k) = \det(I_n - \rho(g) D^k)
		\end{equation}
		for every \(t \in \RR\). The same holds when we replace \(a_+^t\) with any \(u \in \mathfrak{A}\), where \(\mathfrak{A}\) is defined as in \zcref{lem:positivenormalisation}. Indeed, if \(u = \prod_{i=1}^s (a_i)_+^{t_i}\) with \(a_i \in A\) and \(t_i \in \QQ\), we can write \(t_i = n_i/d\) with \(d\) a common positive denominator. Simultaneous diagonalisation then gives
		\[
		u = \left(\prod_{i=1}^s a_i^{n_i}\right)_+^{1/d},
		\]
		so \zcref[nocap,abbrev]{eqn:elemposparts1} applies. For arbitrary real \(t_i\), we use that
		the continuous function
		\[ \RR^s \to \RR \colon (t_1,\ldots,t_s) \mapsto \det\left( I_n - \prod_{i=1}^s (a_i)_+^{t_i} \rho(g) D^k \right)\]
		equals the constant \(\det(I_n - \rho(g) D^k)\) on \(\QQ^s\), which is dense, and so it equals this value on all of \(\RR^s\). In other words, we have now extended \zcref[nocap,abbrev]{eqn:elemposparts1} to
		\begin{equation}
			\label{eqn:elemposparts2}
			\det(I_n - u \rho(g) D^k) = \det(I_n - \rho(g) D^k)
		\end{equation}
		for every \(u \in \mathfrak{A}\).
		
		Finally, \zcref{lem:positivenormalisation} gives us that \(\tilde{\rho}(g) \tilde{D}^k = \pi(\rho(g) D^k) = u \rho(g) D^k\) for some \(u \in \mathfrak{A}\). Together with \zcref[nocap,abbrev]{eqn:elemposparts2} this implies that
		\[
		\det(I_n-\tilde{\rho}(g)\tilde{D}^k) = \det(I_n - u\rho(g) D^k) = \det(I_n - \rho(g) D^k). \qedhere
		\]
	\end{proof}
	
	\subsection{Signed determinant averages}
	
	The next \zcref[noref,nocap]{lem:absvaluetoconstants} lets us get rid of the absolute value around each determinant by extracting its sign as a factor of the form \(\delta\epsilon^k\chi(g)\). This will reduce the absolute determinant averages in \zcref{prop:generalaverages} to the signed averages in \zcref{lem:commutingsignedaverages}.
	
	\begin{lemma}
		\label{lem:absvaluetoconstants}
		Let \(G\) be a group, let \(\varphi \in \End(G)\), let \(\rho \colon G \to \GL_n(\RR)\) be a homomorphism such that \(\rho(G)\) is bounded, and let \(D \in \RR^{n\times n}\) such that \(D \rho(g) = \rho(\varphi(g)) D\) for every \(g \in G\). There exist constants \(\delta, \epsilon \in \{-1,1\}\) and a homomorphism \(\chi\colon G \to \{-1,1\}\) such that \(\chi \varphi = \chi\) and
		\begin{equation}
			\label{eqn:absdetremoval}
			\abs{\det(I_n - \rho(g)D^k)} = \delta \epsilon^k \chi(g) \det(I_n - \rho(g)D^k)
		\end{equation}
		for every \(g \in G\) and every \(k \geq 1\).
	\end{lemma}
	\begin{proof}
		We follow the decomposition and sign arguments from \cite[Sec.~3]{dd15-a}, replacing the use of finiteness by boundedness. The proof is divided into three steps.
		
		\medskip
		
		\noindent\textbf{Step 1.} First, we will decompose \(\rho(g)\) and \(D\) into two blocks. Let \(V_{\leq 1}\) and \(V_{>1}\) be the real subspaces corresponding to the generalised eigenspaces of \(D\) for eigenvalues of modulus \(\leq 1\) and \(>1\), respectively. We start by showing that \(v \in V_{\leq 1}\) if and only if \(\|D^r v\|\) is bounded by a polynomial in \(r\). We use the Euclidean norm and its induced matrix norm. Let \(J\) be the Jordan normal form of \(D\) and let \(P \in \GL_n(\CC)\) such that \(D = PJP^{-1}\). Set \(w \coloneq P^{-1}v\).
		
		A Jordan block \(J_\lambda\) of size \(m\) is of the form \(\lambda I_m + U\) with \(U\) an upper triangular matrix with zeroes on its diagonal. Hence \(U^m = 0\), and the binomial theorem gives
		\[ J_\lambda^r = \sum_{j=0}^{\min\{r,m-1\}} \binom{r}{j} \lambda^{r-j} U^j \]
		for every \(r \geq 0\).
		If \(\abs{\lambda} \leq 1\), then \(\abs{\lambda^{r-j}} \leq 1\) for every term in this sum, and each coefficient \(\binom{r}{j}\) is a polynomial in \(r\). Using the triangle inequality, we find that
		\[
		\|J_\lambda^r\|
		\leq \sum_{j=0}^{\min\{r,m-1\}} \binom{r}{j} \abs{\lambda}^{r-j} \|U\|^j
		\leq \sum_{j=0}^{m-1} \frac{r^j}{j!} \|U\|^j
		\leq \left( \sum_{j=0}^{m-1} \frac{\|U\|^j}{j!} \right) (1+r)^{m-1},
		\]
		i.e.\@ the powers of this block are bounded by a polynomial \(C(1+r)^{m-1}\) with \(C\) and \(m\) constants. Let \(J_{\lambda_1}, \ldots, J_{\lambda_s}\) be the Jordan blocks with eigenvalues of modulus \(\leq 1\), and let \(w_i\) be the subvectors of \(w\) corresponding to these blocks. If \(v \in V_{\leq 1}\), then
		\begin{align*}
			\|D^r v\|
			&= \| P J^r w \|\\
			&\leq \|P\| \sum_{i=1}^s \| J_{\lambda_i}^r w_i\| \\
			&\leq \|P\| \sum_{i=1}^s C_i \|w_i\| (1+r)^{m_i - 1},
		\end{align*}
		for fixed \(C_i\), \(m_i\) and \(w_i\), so \(\|D^r v\|\) is indeed bounded by a polynomial in \(r\).
		
		Conversely, suppose \(v \notin V_{\leq 1}\). Then \(w\) has a non-zero subvector \(w' = (w_1, \ldots, w_m)\) in a Jordan block \(J_\lambda = \lambda I_m + U\) with \(\abs{\lambda} > 1\). Let \(s\) be the largest index such that \(w_s \neq 0\) and let \(t \in \NN_0\). Since \(U\) has ones on its superdiagonal, the \(s\)-th coordinate of \(U^t w'\) is either \(w_{s+t}\) or \(0\), depending on whether \(s + t \leq m\). By the choice of \(s\), this coordinate is \(0\) for every \(t \geq 1\), and hence the binomial expansion gives
		\((J_\lambda^r w')_s = \lambda^r w_s\). Therefore,
		\[
		\|P^{-1}\| \| D^r v\|
		\geq \| P^{-1} D^r v\|
		= \| J^r w \|
		\geq \| J_\lambda^r w' \|
		\geq \abs{\lambda}^r \abs{w_s},
		\]
		so \(\| D^r v\| \geq C \abs{\lambda}^r\) for some constant \(C > 0\). Hence, \(\|D^r v\|\) cannot be bounded above by a polynomial in \(r\).
		
		Now choose \(B > 0\) such that \(\|\rho(h)\| \leq B\) for every \(h \in G\). Fix an element \(g \in G\) and a vector \(v \in V_{\leq 1}\). By induction, the matrix relation gives us that \(D^r\rho(g) = \rho(\varphi^r(g))D^r\), and therefore
		\[
		\|D^r\rho(g)v\|
		= \|\rho(\varphi^r(g))D^r v\|
		\leq B\|D^r v\|.
		\]
		As the right-hand side is bounded by a polynomial in \(r\), we get that \(\rho(g)v \in V_{\leq 1}\). Thus, \(V_{\leq 1}\) is \(\rho(G)\)-invariant. With respect to a basis adapted to \(\RR^n = V_{\leq 1} \oplus V_{>1}\), we obtain the decomposition
		\[
			\label{eqn:triangularbasis}
			\rho(g) = \begin{pmatrix}
				\rho_{\leq 1}(g) & \ast \\
				0 & \rho_{>1}(g)
			\end{pmatrix},
			\qquad
			D = \begin{pmatrix}
				D_{\leq 1} & 0 \\
				0 & D_{>1}
			\end{pmatrix}.
		\]
		Both \(\rho_{\leq 1}\) and \(\rho_{>1}\) are homomorphisms, and both their images are bounded. The eigenvalues of \(D_{\leq 1}\) and \(D_{>1}\) have modulus \(\leq 1\) and \(>1\), respectively. We will refer to the ``\(\leq 1\)''- and ``\(>1\)''-subspaces and submatrices as the \emph{non-expanding} and \emph{expanding} parts, respectively.
		
		We now remove the \(\ast\)-block in \(\rho(g)\) by defining
		\[
		\tilde{\rho}(g) \coloneq \begin{pmatrix}
			\rho_{\leq 1}(g) & 0 \\
			0 & \rho_{>1}(g)
		\end{pmatrix}.
		\]
		Multiplying block upper triangular matrices shows that \(\tilde{\rho}\) is a homomorphism with bounded image, and that
		\[
		D\tilde{\rho}(g) = \tilde{\rho}(\varphi(g))D,
		\qquad
		\det(I_n-\rho(g)D^k) = \det(I_n-\tilde{\rho}(g)D^k)
		\]
		for every \(g \in G\) and every \(k \geq 1\). Thus, replacing \(\rho\) by \(\tilde{\rho}\), we may assume that both \(\rho(g)\) and \(D\) are block diagonal.
		
		\medskip
		\noindent\textbf{Step 2.} In this second step, we study the eigenvalues of \(\rho(g)D^k\) corresponding to expanding and non-expanding parts. Fix \(g \in G\) and \(k \geq 1\). For every \(r \geq 1\), we set \(g_r \coloneq g\varphi^k(g)\cdots\varphi^{(r-1)k}(g)\), such that
		\[ (\rho(g)D^k)^r = \rho(g_r)D^{kr}. \]
		On the non-expanding part, we get that
		\[
		\| (\rho_{\leq 1}(g) D_{\leq 1}^k)^r \|
		= \|\rho_{\leq 1}(g_r) D_{\leq 1}^{kr} \|
		\leq \|\rho_{\leq 1}(g_r)\| \|D_{\leq 1}^{kr}\|,
		\]
		where on the right-hand side the first factor is bounded by a constant and the second factor is bounded by a polynomial in \(r\). Taking a non-zero complex eigenvector with eigenvalue \(\lambda\), we find that \(\abs{\lambda}^r\) is bounded by a polynomial in \(r\). This is only possible if \(\lambda\), and by extension every eigenvalue of \(\rho_{\leq 1}(g)D_{\leq 1}^k\), has modulus at most one.
		
		Conversely, on the expanding part, for every non-zero complex vector \(v\), we get that
		\[
		\|(\rho_{>1}(g)D_{>1}^k)^r v\|
		= \|\rho_{>1}(g_r)D_{>1}^{kr}v\|
		\geq \frac{\|D_{>1}^{kr}v\|}{\|\rho_{>1}(g_r)^{-1}\|}.
		\]
		The denominator is bounded by a constant, since \(\rho_{>1}(g_r)^{-1} = \rho_{>1}(g_r^{-1})\), and we know from the previous step that the numerator grows exponentially in \(r\). By taking \(v\) to be an eigenvector, we conclude that every eigenvalue of \(\rho_{>1}(g)D_{>1}^k\) has modulus greater than one.
		
		\medskip
		
		\noindent\textbf{Step 3.} In the third and final step, we get rid of the absolute value around the determinant. We define \(\chi\) as the homomorphism
		\[ \chi \colon G \to \{-1,1\} \colon g \mapsto \det \rho_{>1}(g). \]
		To see that this is well-defined, fix \(g \in G\). For every \(r \in \ZZ\), we have
		\[\abs{\chi(g)}^r = \abs{\det\rho_{>1}(g)}^r = \abs{\det\rho_{>1}(g)^r} = \abs{\det\rho_{>1}(g^r)}, \] 
		and the right-hand side is bounded. So \(\abs{\chi(g)} = 1\), and since we are dealing with real matrices, the determinant is real and hence \(\chi(g) \in \{-1,1\}\). Taking the determinant on both sides of the equality
		\[ D_{>1} \rho_{>1}(g) = \rho_{>1}(\varphi(g)) D_{>1} \]
		and dividing both sides by \(\det D_{>1} \neq 0\), we get that \(\chi \varphi = \chi\).
		
		Let \(p\) and \(q\) be the numbers of real eigenvalues of \(D\) that are greater than \(1\) and smaller than \(-1\), respectively, counted with algebraic multiplicity. We set
		\[
		\delta \coloneq (-1)^{p+q},
		\qquad
		\epsilon \coloneq (-1)^q,
		\]
		such that
		\[ \delta \epsilon^k = \begin{cases}
			(-1)^p & \text{ if \(k\) is odd},\\
			(-1)^{p+q} & \text{ if \(k\) is even}.
		\end{cases} \]
		Then, applying \cite[Lem.~4.4]{dd15-a} to \(D_{>1}\), and using that every eigenvalue of \(D_{>1}^k\) has modulus greater than one, gives us that
		\[
		\delta \epsilon^k \det(I-D_{>1}^k) > 0.
		\]
		Every eigenvalue of \((\rho_{>1}(g)D_{>1}^k)^{-1}\) has modulus less than one, and the same holds for \(D_{>1}^{-k}\), hence we find that
		\[
		\det(I-(\rho_{>1}(g)D_{>1}^k)^{-1}) > 0,
		\qquad
		\det(I-D_{>1}^{-k}) > 0.
		\]
		Using the identity \(I-M = -M(I-M^{-1})\) for an invertible matrix \(M\), we obtain
		\[
		\frac{\det(I-\rho_{>1}(g)D_{>1}^k)}{\det(I-D_{>1}^k)}
		= \chi(g)\frac{\det(I-(\rho_{>1}(g)D_{>1}^k)^{-1})}{\det(I-D_{>1}^{-k})},
		\]
		hence
		\[
		\chi(g) \det(I-D_{>1}^k) \det(I-\rho_{>1}(g)D_{>1}^k) > 0.
		\]
		Combining this with \(\delta\epsilon^k\det(I-D_{>1}^k) > 0\) gives
		\[
		\delta \epsilon^k \chi(g) \det(I-\rho_{>1}(g)D_{>1}^k) > 0.
		\]
		For the non-expanding part we instead find that
		\[
		\det(I - \rho_{\leq 1}(g)D_{\leq 1}^k) \geq 0.
		\]
		Once again combining the previous two inequalities and using the block decomposition from the first step gives
		\[
		\delta\epsilon^k \chi(g) \det(I_n-\rho(g)D^k)
		= \delta\epsilon^k \chi(g) \det(I-\rho_{>1}(g)D_{>1}^k) \det(I - \rho_{\leq 1}(g)D_{\leq 1}^k)
		\geq 0,
		\]
		which proves \zcref[nocap,abbrev]{eqn:absdetremoval}.
	\end{proof}
	
	Since the reduction to bounded representations doesn't change the determinants, the sign formula obtained above will also hold for the integer matrices obtained by applying \zcref{lem:matrixreduction}. Therefore, we can use the resulting homomorphism \(\chi\) in the lemma below.
	
	\begin{lemma}
		\label{lem:commutingsignedaverages}
		Let \(G\) be a group, let \(N\) be a finite index, normal subgroup of \(G\), let \(\varphi \in \End(G)\) such that \(\varphi(N) \leq N\), and let \(\rho \colon G \to \GL_n(\ZZ)\) be a homomorphism such that \(\rho(N)\) is simultaneously diagonalisable over \(\CC\). Suppose that \(D \in \ZZ^{n \times n}\) satisfies \(D\rho(g) = \rho(\varphi(g))D\) for every \(g \in G\), and let \(\chi \colon G \to \{-1,1\}\) be a homomorphism such that \(\chi\varphi = \chi\). If \(\chi(g)\det(I_n - \rho(g)D^k)\) depends only on the coset \(gN\), for every \(k \geq 1\), then there exist integer matrices \(M_0,\ldots,M_n\) such that
		\[
		\frac{1}{\ind{G}{N}}\sum_{gN \in G/N}\chi(g)\det(I_n - \rho(g)D^k)
		= \sum_{j=0}^n (-1)^j\tr(M_j^k)
		\]
		for every \(k \geq 1\).
	\end{lemma}
	\begin{proof}
		For every \(j \in \{0,\ldots,n\}\), set \(d_j \coloneq \binom{n}{j}\) and
		\[
		V_j \coloneq \QQ^{d_j},
		\qquad
		X_j \coloneq \comp{j}{D},
		\qquad
		Y_j(g) \coloneq \chi(g)\comp{j}{\rho(g)}.
		\]
		From the multiplicativity of compound matrices and from \(\chi\varphi=\chi\), we find that \(Y_j\) is a homomorphism and that \(X_j Y_j(g) = Y_j(\varphi(g)) X_j\). Applying compound matrices to a common diagonalisation of \(\rho(N)\) also shows that \(Y_j(N)\) is simultaneously diagonalisable.
		We define the \(\QQ\)-vector space
		\[
		W_j \coloneq \Span_{\QQ} \set{ (I_{d_j} - Y_j(a))v }{ a \in N,\  v \in V_j}
		\]
		for every \(j\). Using the identities below,
		\begin{align*}
			X_j(I_{d_j}-Y_j(a)) &= (I_{d_j}-Y_j(\varphi(a)))X_j,\\
			Y_j(g)(I_{d_j}-Y_j(a)) &=(I_{d_j}-Y_j(gag^{-1}))Y_j(g),
		\end{align*}
		we find that \(W_j\) is invariant under \(X_j\) and \(Y_j(G)\). Let \(\bar{X}_j\) and \(\bar{Y}_j(g)\) denote the induced maps on \(V_j/W_j\), and in particular \(\bar{Y}_j(a) = I\) for \(a \in N\).
		
		Now, for each \(j\) choose a basis of \(W_j\) of the form 
		\[ \set{ (I_{d_j} - Y_j(a_{i,j}))v_i }{ i = 1, \ldots, s_j },\]
		and to keep things simple, let \(a_1, \ldots, a_s\) enumerate all of the \(a_{i,j}\). Then
		\[ W_j=\sum_{i=1}^s(I-Y_j(a_i))V_j\]
		for every \(0 \leq j \leq n\).
		
		We take the product of the characteristic polynomials of all \(Y_j(a_i)\), remove all factors \(x-1\), and finally divide by the resulting polynomial's value at \(1\). This gives a polynomial \(p \in \QQ[x]\) such that \(p(1) = 1\), and such that every other eigenvalue of a matrix \(Y_j(a_i)\) is a root of this polynomial. Set
		\[ P_j \coloneq \prod_{i=1}^s p(Y_j(a_i)). \]
		We established before that \(Y_j(N)\) is simultaneously diagonalisable, so we may work in a common eigenbasis. For an eigenvector \(v\) of \(Y_j(a_i)\) with eigenvalue \(\lambda\), we find
		\[
		p(Y_j(a_i))(I-Y_j(a_i))v = p(\lambda)(1-\lambda)v = 0,
		\]
		since by construction \(p(\lambda) = 0\) whenever \(\lambda \neq 1\). Thus \(p(Y_j(a_i))(I-Y_j(a_i)) = 0\).
		Both factors here are polynomials in \(Y_j(a_i)\), so they commute. Then
		\[
		P_j(I - Y_j(a_i)) = \left( \prod_{t \neq i} p(Y_j(a_t)) \right) p(Y_j(a_i))(I-Y_j(a_i)) = 0,
		\]
		so \(P_jW_j = 0\). Since \(\bar{Y}_j(a_i)=I\) and \(p(1)=1\), the map induced by \(P_j\) on \(V_j/W_j\) is the identity. Expanding \(P_j\) and using multiplicativity of the \(Y_j\), we find elements \(b_1, \ldots, b_r \in N\) (each a product \(\prod_i a_i^{e_i}\)) and constants \(c_1, \ldots, c_r \in \QQ\) such that 
		\[
		P_j=\sum_{t=1}^r c_t Y_j(b_t),
		\qquad \sum_{t=1}^r c_t = p(1)^s =1,
		\]
		for every \(j\).
		Fix \(g\in G\) and \(k\geq1\). In a basis of \(V_j\) extending a basis of \(W_j\), we have that
		\[
		P_j = \begin{pmatrix}
			0 & \ast \\
			0 & I
		\end{pmatrix},
		\qquad
		Y_j(g)X_j^k = \begin{pmatrix}
			\ast & \ast \\
			0 & \bar{Y}_j(g)\bar{X}_j^k
		\end{pmatrix},
		\]
		so their product has trace \(\tr(\bar{Y}_j(g)\bar{X}_j^k)\). Since \(b_t gN = gN\), the hypothesis and \zcref{lem:detalttracesum} give
		\begin{align*}
			\chi(g)\det(I_n-\rho(g)D^k)
			&=\sum_{t=1}^r c_t \chi(b_t g)\det(I_n-\rho(b_t g)D^k)\\
			&=\sum_{j=0}^n(-1)^j\tr(P_jY_j(g)X_j^k)\\
			&=\sum_{j=0}^n(-1)^j\tr(\bar{Y}_j(g)\bar{X}_j^k).
		\end{align*}
		By \zcref{lem:integerbase}, the matrices \(\bar{X}_j\), \(\bar{Y}_j(g)\) and \(\bar{Y}_j(g^{-1})\) have integer entries in suitable quotient bases.
		
		Let \(\bar{\varphi}\) denote the induced endomorphism on \(G/N\). For each \(j\), we define
		\[
		\tilde{Y}_j \colon G/N \to \GL(V_j/W_j) \colon gN \mapsto \bar{Y}_j(g).
		\]
		This is well-defined, since for every \(a \in N\), we have that \(\bar Y_j(ga)=\bar Y_j(g)\bar Y_j(a)=\bar Y_j(g)\). Then
		\[ 
		\bar{X}_j \tilde{Y}_j(gN)
		= \bar{X}_j\bar{Y}_j(g)
		= \bar{Y}_j(\varphi(g)) \bar{X}_j
		= \tilde{Y}_j(\bar{\varphi}(gN)) \bar{X}_j
		\]
		for every \(j\) and every \(gN \in G/N\). Thus:
		\begin{align*}
			&\frac{1}{\ind{G}{N}} \sum_{gN\in G/N}\chi(g)\det(I_n-\rho(g)D^k) \\
			&\qquad = \frac{1}{\ind{G}{N}} \sum_{gN\in G/N} \sum_{j=0}^n (-1)^j \tr( \tilde{Y}_j(gN) \bar{X}_j^k )\\
			&\qquad=\sum_{j=0}^n(-1)^j \left( \frac{1}{\ind{G}{N}} \sum_{gN\in G/N} \tr(\tilde{Y}_j(gN)\bar{X}_j^k)\right).
		\end{align*}
		Applying \zcref{lem:traceavgs} now gives the required integer matrices \(M_0, \ldots, M_n\).
	\end{proof}
	
	\subsection{Proof and corollaries} We now prove \zcref{prop:generalaverages} by combining the matrix reductions with the sign and trace formulas.
	
	\begin{proof}[Proof of \zcref{prop:generalaverages}]
		By \zcref{lem:matrixreduction}, we may assume that \(D\) is invertible and \(\rho(N)\) is simultaneously diagonalisable over \(\CC\). Applying \zcref{lem:absvaluetoconstants} to the homomorphism \(\tilde{\rho}\) and matrix \(\tilde{D}\) obtained from \zcref{lem:boundedreduction}, and using that this preserves determinants, we find that there exist constants \(\delta,\epsilon \in \{-1,1\}\) and a homomorphism \(\chi \colon G \to \{-1,1\}\) such that \(\chi\varphi = \chi\) and
		\[
		\abs{\det(I_n - \rho(g)D^k)} = \delta\epsilon^k\chi(g)\det(I_n - \rho(g)D^k)
		\]
		for every \(g \in G\) and every \(k \geq 1\). This also shows that \(\chi(g)\det(I_n-\rho(g)D^k)\) depends only on \(gN\). We then invoke \zcref{lem:commutingsignedaverages} to obtain integer matrices \(M_0, \ldots, M_n\) such that
		\[ A_k = \delta\epsilon^k\sum_{j=0}^n (-1)^j\tr(M_j^k). \]
		Hence the \(A_k\) are integers, and by \zcref{cor:zetafuncoftraces} their associated zeta function is given by
		\[ Z_A(z) = \prod_{j=0}^n \det(I-\epsilon zM_j)^{\delta(-1)^{j+1}}, \]
		which exists near zero and is rational.
	\end{proof}
	
	While \(\tau\) is defined as the product of multiple absolute values of determinants, we can use block diagonal matrices to write it as the absolute value of a single determinant. This lets us apply \zcref{prop:generalaverages} to averages of \(\tau\).
	
	\begin{corollary}
		\label{cor:tauaverages}
		Let \(G\) be a group, let \(\varphi \in \End(G)\), and let \(N\) be a finite index, normal \NR-subgroup of \(G\) such that \(\varphi(N) \leq N\). For every \(k\geq1\), the quantity
		\[ A_k\coloneq\frac{1}{\ind{G}{N}}\sum_{gN\in G/N}\tau(\restr{\iota_g\varphi^k}{N}) \]
		is an integer and the sequence \((A_k)_{k \in \NN}\) has a rational associated zeta function.
	\end{corollary}
	\begin{proof}
		Let \(\fact{0}{N},\ldots,\fact{c}{N}\) be the canonical factors of \(N\). Let \(D_i\) and \(\rho_i(g)\) be the integer matrices induced on the \(i\)-th factor by \(\varphi\) and \(\iota_g\), respectively. Set
		\[ D \coloneq \diag(D_0, \ldots, D_c),\qquad
		\rho(g) \coloneq \diag(\rho_0(g), \ldots, \rho_c(g)). \]
		By \zcref{def:tau} and the identity \(\varphi\iota_g=\iota_{\varphi(g)}\varphi\), we find that
		\[ \tau(\restr{\iota_g\varphi^k}{N}) = \abs{\det(I-\rho(g)D^k)},
		\qquad D\rho(g)=\rho(\varphi(g))D. \]
		Set \(M \coloneq \sqrt[N]{[N,N]}\). Conjugation by \(M\) is trivial on \(N/M\) and on the factors of the adapted lower central series of \(M\), hence \(\rho(M) = 1\), and \(\rho(N)\) is abelian. By \zcref{lem:invariances}(i), \(\abs{\det(I-\rho(g)D^k)}\) depends only on the coset \(gN\), and the result then follows from \zcref{prop:generalaverages}.
	\end{proof}
	
	\begin{corollary}
		\label{cor:Tisinteger}
		Let \(G\) be a virtually polycyclic group and let \(\varphi\in\End(G)\). Then \(\TR(\varphi)\) is an integer.
	\end{corollary}
	\begin{proof}
		Choose \(N\) as in \zcref{def:tamedReidNr}. For every \([t]_G\in\Fix(P_\varphi)\), the subgroup \(D_t=C_t\cap N\) is an \NR-group by \zcref{lem:admissiblesubs}. Applying \zcref{cor:tauaverages} with \(k=1\), we find that every outer summand in \zcref[nocap,abbrev]{eqn:TRformula} is an integer, hence so is their sum.
	\end{proof}

	\section{Rationality of the tamed Reidemeister zeta function}
	\label{sec:rationalityZeta}
	
	Armed with the determinant averages from the previous section, we are now ready to define and prove the rationality of the tamed Reidemeister zeta function. Following the approach of Fel'shtyn and Hill in \cite[Thm.~6]{fh94-a}, we group the conjugacy classes of torsion elements into periodic orbits of \(P_\varphi\), and show that the classes in each orbit all have the same summand in \zcref[nocap,abbrev]{eqn:TRformula}. This way, we show that \(\TR_\varphi(z)\) is a product, taken over the periodic orbits, of rational functions, and is therefore rational itself.
	
	\begin{definition}
		\label{def:tamedReidZeta}
		Let \(G\) be a virtually polycyclic group and let \(\varphi \in \End(G)\). The \emph{tamed Reidemeister zeta function} of \(\varphi\) is the zeta function
		\[
		\TR_\varphi(z)\coloneq\exp\sum_{k=1}^\infty\frac{\TR(\varphi^k)}{k}z^k.
		\]
		The formal power series in the exponent converges for \(z \in \CC\) with sufficiently small modulus, as established in \zcref{thm:tamedrationality} below.
	\end{definition}
	
	\subsection{Periodic orbits}
	
	\begin{lemma}
		\label{lem:nilptauswap}
		Let \(G, H\) be two \NR-groups, and let \(\varphi \colon G \to H\) and \( \psi \colon H \to G\) be two homomorphisms. Then \(\tau(\psi \varphi) = \tau(\varphi \psi)\).
	\end{lemma}
	\begin{proof} 
		Since \(\varphi(\psi \varphi) = (\varphi \psi) \varphi\) and \(\psi(\varphi \psi) = (\psi \varphi) \psi\), we get induced maps
		\begin{align*}
			\Phi &\colon \R[\psi \varphi] \to \R[\varphi \psi] \colon [g]_{\psi \varphi} \mapsto [\varphi(g)]_{\varphi \psi},\\
			\Psi &\colon \R[\varphi \psi] \to \R[\psi \varphi] \colon [h]_{\varphi \psi} \mapsto [\psi(h)]_{\psi \varphi}.
		\end{align*}
		For any \(\theta \in \End(G)\) and any \(x \in G\), we have that \([x]_\theta = [\theta(x)]_\theta\). Thus, we find that
		\[ (\Psi\Phi)([g]_{\psi \varphi}) = [(\psi\varphi)(g)]_{\psi\varphi} = [g]_{\psi\varphi},\]
		and similarly \(\Phi\Psi\) is the identity on \(\R[\varphi\psi]\). Hence \(R(\psi \varphi) = R(\varphi\psi)\), and by \zcref{prop:NRtauequalsRwhenfinite} then \(\tau(\psi \varphi) = \tau(\varphi\psi)\) as well.
	\end{proof}
	
	\begin{lemma}
		\label{lem:tauaverageswapping}
		Let \(C,C'\) be groups and let \(D \normalsub C\) and \(D' \normalsub C'\) be \NR-subgroups of finite index. Let
		\[ \alpha \colon C \to C', \qquad \beta \colon C' \to C\]
		be homomorphisms such that \(\alpha(D) \leq D'\) and \(\beta(D') \leq D\). Then
		\begin{align*}
			&\frac{1}{\ind{C}{D}} \sum_{cD \in C/D} \tau( \restr{\iota_c\beta\alpha}{D})\\
			&\qquad = \frac{1}{\ind{C'}{D'}} \sum_{c'D' \in C'/D'} \tau( \restr{\iota_{c'}\alpha\beta}{D'}).
		\end{align*}
	\end{lemma}
	\begin{proof}
		For \(c \in C\), \(c' \in C'\), we apply \zcref{lem:nilptauswap} to the homomorphisms \(\restr{\iota_c\beta}{D'} \colon D' \to D\) and \(\restr{\iota_{c'}\alpha}{D} \colon D \to D'\). This gives us that
		\[ \tau(\restr{\iota_{c\beta(c')}\beta\alpha}{D}) = \tau(\restr{\iota_{c'\alpha(c)}\alpha\beta}{D'}).\]
		For a fixed \(c'\), the map \(cD \mapsto c\beta(c')D\) is a bijection of \(C/D\), and similarly for the map
		\(c'D' \mapsto c'\alpha(c)D'\) with \(c\) fixed. Therefore,
		\begin{align*}
			\frac{1}{\ind{C}{D}} \sum_{cD \in C/D} \tau( \restr{\iota_c\beta\alpha}{D})
			&= 
			\frac{1}{\ind{C'}{D'}}  \sum_{c'D' \in C'/D'} \frac{1}{\ind{C}{D}} \sum_{cD \in C/D} \tau( \restr{\iota_{c\beta(c')}\beta\alpha}{D}) \\
			&= 
			\frac{1}{\ind{C}{D}} \sum_{cD \in C/D} \frac{1}{\ind{C'}{D'}} \sum_{c'D' \in C'/D'} \tau(\restr{\iota_{c'\alpha(c)}\alpha\beta}{D'}) \\
			&= 
			\frac{1}{\ind{C'}{D'}} \sum_{c'D' \in C'/D'} \tau(\restr{\iota_{c'}\alpha\beta}{D'}).\qedhere
		\end{align*}
	\end{proof}
	
	We now apply these results to the periodic orbits of \(P_\varphi\).
	
	\begin{proposition}
		\label{prop:periodicaverages}
		Let \(G\) be a virtually polycyclic group, let \(\varphi \in \End(G)\), and let \(\mathcal{O}\) be a periodic orbit of \(P_\varphi\) of length \(l\). There is a sequence \((A_{k,\mathcal{O}})_{k \in \NN}\) of integers such that every class in \(\mathcal{O}\) has summand \(A_{k,\mathcal{O}}\) in the definition of \(\TR(\varphi^{kl})\). Moreover, its associated zeta function
		\[
		Z_{\mathcal{O}}(z) \coloneq \exp\sum_{k=1}^\infty \frac{A_{k,\mathcal{O}}}{k} z^k
		\]
		exists and is rational.
	\end{proposition}
	\begin{proof}
		Let \(N\) be as in \zcref{def:tamedReidNr}. We will start by comparing the summands of \([t]_G\) and \([\varphi(t)]_G\) in \(\TR(\varphi^m)\), for some \(m \in \NN\) and some conjugacy class \([t]_G \in \Fix(P_{\varphi}^m)\). We set \(t' \coloneq \varphi(t)\) and choose \(g \in G\) such that \(t = g\varphi^m(t)g^{-1}\).
		
		Consider the homomorphisms
		\[
		\alpha \coloneq \restr{\varphi}{C_t} \colon C_t \to C_{t'},
		\qquad
		\beta \coloneq \restr{\iota_g\varphi^{m-1}}{C_{t'}} \colon C_{t'} \to C_t,
		\]
		and note that \(\alpha(D_t) \leq D_{t'}\) and \(\beta(D_{t'}) \leq D_t\). Composing them both ways gives
		\[
		\beta \alpha = \restr{\iota_g \varphi^m}{C_t}, 
		\qquad
		\alpha \beta = \restr{\iota_{\varphi(g)}\varphi^m}{C_{t'}}.
		\]
		Applying the previous \zcref[nocap,noref]{lem:tauaverageswapping} gives us that the summands of \([t]_G\) and \([t']_G\) are one and the same. Continuing this inductively, the summands of \([t]_G\) and \([\varphi^i(t)]_G\) are the same for every \(i \in \{1, \ldots, m-1\}\).
		
		Now fix \([t]_G \in \mathcal{O}\), choose \(g \in G\) such that \(t = g\varphi^l(t)g^{-1}\) and set \(\psi \coloneq \iota_g\varphi^l\). Then \(\psi(t) = t\), so both \(C_t\) and \(D_t\) are \(\psi\)-invariant. We define
		\[
		A_{k,\mathcal{O}} \coloneq \frac{1}{\ind{C_t}{D_t}} \sum_{cD_t \in C_t/D_t} \tau(\restr{\iota_c\psi^k}{D_t}).
		\]
		We set \(h \coloneq g\varphi^l(g)\cdots \varphi^{(k-1)l}(g)\) such that \( \psi^k = \iota_h \varphi^{kl}\). Since \(\psi^k(t) = t\), \(A_{k,\mathcal{O}}\) is the summand of \([t]_G\) in \(\TR(\varphi^{kl})\), and hence also that of every other class in \(\mathcal{O}\). Finally, applying \zcref{cor:tauaverages} gives that \(A_{k,\mathcal{O}}\) is an integer and that \(Z_{\mathcal{O}}(z)\) exists and is rational.
	\end{proof}
	
	\subsection{Rationality}
	
	\begin{theorem}
		\label{thm:tamedrationality}
		Let \(G\) be a virtually polycyclic group and let \(\varphi \in \End(G)\). Its tamed Reidemeister zeta function \(\TR_\varphi(z)\) exists and can be decomposed as
		\[
		\TR_\varphi(z) = \prod_{\mathcal{O}}Z_{\mathcal{O}}( z^{l(\mathcal{O})} ),
		\]
		where the product runs over the periodic orbits of \(P_\varphi\), \(l(\mathcal{O})\) is the length of the orbit \(\mathcal{O}\), and \(Z_{\mathcal{O}}(z)\) is defined as in \zcref{prop:periodicaverages}. Consequently, \(\TR_\varphi(z)\) is rational.
	\end{theorem}
	\begin{proof}
		Since \(P_{\varphi^k}=P_\varphi^k\), a conjugacy class \([t]_G\) belongs to \(\Fix(P_{\varphi^k})\) precisely when it lies in a periodic orbit of \(P_\varphi\) whose length divides \(k\). Hence \zcref{prop:periodicaverages} gives
		\begin{equation}
			\label{eqn:Tperiodicorbits}
			\TR(\varphi^k)=
			\sum_{\substack{\mathcal{O} \text{ periodic}\\ l(\mathcal{O}) \mid k}}
			l(\mathcal{O})A_{k/l(\mathcal{O}),\mathcal{O}}.
		\end{equation}
		By \zcref{prop:periodicaverages}, each sequence \((A_{m,\mathcal O})_{m \in \NN}\) has an associated zeta function. Thus, for each periodic orbit \(\mathcal{O}\), there exist constants \(C_{\mathcal{O}}>0\) and \(E_{\mathcal{O}} \geq 1\) such that for every \(m \in \NN\), \(\abs{ A_{m,\mathcal{O}}} \leq C_{\mathcal{O}} E_{\mathcal{O}}^m\). There are only finitely many periodic orbits \(\mathcal{O}\) by \zcref{prop:segalconjclasses}, so we can take
		\[
		C \coloneq \sum_{\mathcal{O}} l(\mathcal{O})C_{\mathcal{O}},
		\qquad
		E \coloneq \max_{\mathcal{O}} E_{\mathcal{O}}.
		\]
		Using \zcref[nocap,abbrev]{eqn:Tperiodicorbits} we obtain
		\[
		0
		\leq \TR(\varphi^k)
		\leq \sum_{\substack{\mathcal{O} \text{ periodic}\\ l(\mathcal{O}) \mid k}} l(\mathcal{O})C_{\mathcal{O}} E_{\mathcal{O}}^{k/l(\mathcal{O})}
		\leq CE^k.
		\]
		Consequently, the series defining \(\TR_\varphi(z)\) converges absolutely for \(|z|< 1/E\).
		
		Using \zcref[nocap,abbrev]{eqn:Tperiodicorbits} once more, we get that
		\begin{align*}
			\TR_\varphi(z)
			&= \exp \sum_{k=1}^\infty \frac{\TR(\varphi^k)}{k} z^k \\
			&= \exp \sum_{k=1}^\infty \sum_{\substack{\mathcal{O} \text{ periodic}\\ l(\mathcal{O}) \mid k}} l(\mathcal{O}) \frac{A_{k/l(\mathcal{O}),\mathcal{O}}}{k} z^k \\
			&= \exp \sum_{\mathcal{O}} \sum_{m=1}^\infty \frac{A_{m,\mathcal{O}}}{m} z^{ml(\mathcal{O})} \\
			&= \prod_{\mathcal{O}} Z_{\mathcal{O}}(z^{l(\mathcal{O})}).
		\end{align*}
		Each factor is rational by \zcref{prop:periodicaverages}, hence so is their finite product.
	\end{proof}

    At the end of \zcref{sec:TamedRNr}, we mentioned cases in which the tamed Reidemeister number equals the (ordinary) Reidemeister number or the Nielsen number. These comparisons can be extended to the associated zeta functions.
	
	\begin{corollary}
		\label{cor:reidrationality}
		Let \(G\) be a virtually polycyclic group and let \(\varphi \in \End(G)\) be tame. Then \(R_\varphi(z) = \TR_\varphi(z)\), and in particular \(R_\varphi(z)\) is rational.
	\end{corollary}
	\begin{proof}
		Since \(\varphi\) is tame, \(R(\varphi^k) < \infty\) for every \(k\). By \zcref{prop:reidistamereidwhenfinite}, \(R(\varphi^k) = \TR(\varphi^k)\), hence also \(R_\varphi(z) = \TR_\varphi(z)\). Rationality then follows from \zcref{thm:tamedrationality}.
	\end{proof}
	
	\begin{corollary}
		\label{cor:nielsrationality}
		Let \(M\) be a compact infra-solvmanifold and let \(f \colon M \to M\) be a self-map. Then its Nielsen zeta function \(N_f(z)\) is rational.
	\end{corollary}
	\begin{proof}
		By \zcref{prop:nielsistamereidwhentf}, \(N(f^k) = \TR(f_*^k)\) for every \(k \in \NN\). Thus \(N_f(z) = \TR_{f_*}(z)\), and the latter is rational by \zcref{thm:tamedrationality}.
	\end{proof}
	
	We finish this section by relating our results to previous work.
	
	\begin{remark}
		For a finite group \(G\), the product in \zcref{thm:tamedrationality} reduces to the formula of Fel'shtyn and Hill \cite[Thm.~6]{fh94-a}. Indeed, if \(N = 1\), then \(D_t = 1\) and \(\tau(\id_{\{1\}}) = 1\). Hence \(A_{k,\mathcal{O}} = 1\) for every \(k\), and \(Z_{\mathcal{O}}(z) = (1-z)^{-1}\). Thus:
		\[
		R_\varphi(z) = \TR_\varphi(z)
		= \prod_{\mathcal{O}}\frac{1}{1-z^{l(\mathcal{O})}}.
		\]
	\end{remark}
	
	\begin{remark}
		Deré proved that the Reidemeister zeta function of a tame endomorphism of a virtually polycyclic group always coincides with that of a tame, injective endomorphism of a virtually nilpotent group \cite[Cor.~2.4]{dere25-a}. Thus, rationality of Reidemeister zeta functions on virtually polycyclic groups also follows from the combination of \zcref{cor:reidrationality}, restricted to virtually nilpotent groups, and Deré's result.
	\end{remark}

	\section{Future work}
	\label{sec:futurework}
	
	We conclude this paper by suggesting some directions in which the study of tamed Reidemeister numbers can be continued.
	
	A first direction is to extend the results of this paper to bi-twisted conjugacy, which has received attention recently for e.g.\@ finite \cite{st26-a} and nilpotent \cite{fk22-a} groups. Given \(\varphi,\psi \in \End(G)\), the bi-twisted conjugacy classes are the equivalence classes of the relation
	\[
	g_1 \sim_{\varphi,\psi} g_2 \iff \exists h \in G : g_1 = \varphi(h)g_2\psi(h)^{-1}.
	\]
	Their number is called the \emph{coincidence Reidemeister number} \(R(\varphi,\psi)\). Both of the formulas unified by \zcref{thm:Rformula} have ``bi-twisted counterparts''. On the one hand, in \cite[Cor.~3.2]{st26-a}, Senden and the author obtain the following formula for finite groups:
	\[
	R(\varphi,\psi)
	= \sum_{\substack{[t]_G\\{[\varphi(t)]_G=[\psi(t)]_G}}}
	\frac{\card{[t]_G}}{\card{[\varphi(t)]_G}}.
	\]
	On the other hand, averaging formulas for infinite groups are also known; see e.g.\@ \cite[Thm.~4.9]{kl07-a} and \cite[Thm.~4.2]{hlp12-a}. It is therefore natural to ask whether \zcref{thm:Rformula} admits a bi-twisted analogue, which reduces to the aforementioned formulas in their specific settings.
	
	Moreover, \emph{coincidence Reidemeister and Nielsen zeta functions} have also been defined, and here rationality already fails in very simple cases \cite[Ex.~4.5]{fk22-a}. By imposing additional restrictions, though, rationality has been obtained in certain situations; see e.g.\@ \cite[Thm.~3.8(3)]{bbf26-a} and \cite[Sec.~7]{fs26-a}. These results suggest that the study of ``compatibility conditions'' on a (tame) pair of endomorphisms or self-maps may be a worthwhile endeavour. In particular, if the tamed Reidemeister number can be extended to bi-twisted conjugacy, it would be nice to see if these same conditions then guarantee rationality of the tamed Reidemeister zeta function.
	
	\medskip
	
	A second direction is to extend the definition of the tamed Reidemeister number to a larger class of groups. \zcref{def:tamedReidNr} depends on \zcref[nosort]{prop:fullinvsub,prop:segalconjclasses}, and is therefore limited to virtually polycyclic groups. Perhaps we can extend the definition of \(\tau\), and then find a larger class of groups for which \zcref{prop:segalconjclasses} holds and which admit finite index subgroups on which \(\tau\) is then defined. Or, even better, perhaps there is another way of defining \(\TR(\varphi)\) without the use of a formula like \zcref[nocap,abbrev]{eqn:TRformula}, such that it still coincides with \(R(\varphi)\) and \(N(f)\) when relevant.

	\printbibliography
	
\end{document}